\documentclass[a4paper,reqno,draft,12pt]{amsart}
\usepackage{amsmath}
\usepackage{amssymb}
\usepackage{amsthm}
\usepackage{mathrsfs}
\usepackage{times}   
\allowdisplaybreaks
\numberwithin{equation}{section}
\allowdisplaybreaks
\theoremstyle{plain}
 \newtheorem{theorem}{Theorem}[section]
 
 \newtheorem{proposition}[theorem]{Proposition}
 \newtheorem{lemma}[theorem]{Lemma}
\theoremstyle{definition}
 \newtheorem{definition}{Definition}[section]
\theoremstyle{remark}
 \newtheorem{remark}{Remark}
\newcommand{\ep}{\varepsilon}
\newcommand{\p}{\partial}
\newcommand{\RR}{\mathbb{R}}
\newcommand{\TT}{\mathbb{T}}
\begin{document}
\title[Dispersive Flow]{
A third-order dispersive flow \\
for closed curves into K\"ahler manifolds }
\author[E.~Onodera]{Eiji ONODERA}
\address{Mathematical Institute, Tohoku University, Sendai 980-8578, Japan}
\email{sa3m09@math.tohoku.ac.jp}
\thanks{The author is supported by the JSPS Research Fellowships 
        for Young Scientists and 
        the JSPS Grant-in-Aid for Scientific Research No. 19$\cdot$3304.}
\maketitle
\begin{abstract}
This paper is devoted to studying the initial value problem 
for a third-order dispersive equation for closed curves 
into K\"ahler manifolds. 
This equation is a geometric generalization 
of a two-sphere valued system 
modeling the motion of vortex filament. 
We prove the local existence theorem
by using geometric analysis and classical energy method.
\end{abstract}
%
%
\section{Introduction}
\label{section:introduction}
In this paper we study the initial value problem of a third-order 
dispersive flow describing the motion of closed curves on K\"ahler 
manifolds. 
Let $(N, J, g)$ be a K\"ahler manifold 
with an almost complex structure $J$ and a K\"ahler metric $g$, 
and let $\nabla$ be the Levi-Civita connection with respect to $g$.
Consider the initial value problem of the form
\begin{alignat}{2}
  u_{t}
& = 
  a\,\nabla_x^2u_x
  + J_{u}\nabla_xu_x
  +b\,g_{u}(u_x, u_x)u_x
&
  \quad\text{in}\quad
& \RR\times \TT,
\label{equation:pde}
\\
  u(0,x)
& =
  u_0(x)
&
  \quad\text{in}\quad
& \TT,
\label{equation:data}
\end{alignat}
where 
$a, b\in \RR$ are constants, 
$u=u(t,x)$ is an $N$-valued unknown function of 
$(t,x)\in \RR\times \TT$, 
$\TT=\RR/\mathbb{Z}$, 
$u_t(t,x)
=du_{(t,x)}\left(\p/\p{t}\right)_{(t,x)}$, 
$u_x(t,x)
=du_{(t,x)}\left(\p/\p{x}\right)_{(t,x)}$, 
$du$ is the differential of the mapping $u$, 
$\nabla_x
$
is the covariant derivative on $u^{-1}TN$
induced from $\nabla$ with respect to $x$,
and 
$J_u$ and $g_u$ mean 
the almost complex structure and the metric at $u{\in}N$ respectively. 
Here 
$u^{-1}TN=\bigcup_{(t,x)\in \RR\times \TT } T_{u(t,x)}N$
is the pull-back bundle over 
$\RR \times \TT$ from $TN$ via the mapping $u$. 
$V$ is said to be a section of $u^{-1}TN$ over $\TT\times \RR$
if $V(t,x)\in T_{u(t,x)}N$ for all $(t,x)\in \RR\times \TT$. 
$J_u$ and $g_u$ are 
a (1,1)-tensor field and a (0,2)-tensor field along $u$ 
respectively, and  
the equation \eqref{equation:pde} is 
an equality of sections of $u^{-1}TN$. 
We call the solution of \eqref{equation:pde} 
a dispersive flow in this paper. 
In particular, when $a=b=0$, 
the solutions are called one-dimensional Schr\"odinger maps.  
\par
Examples of \eqref{equation:pde} arise in classical mechanics related to 
vortex filament, ferromagnetic spin chain system and etc. 
For $\vec{u}=(u_1,u_2,u_3)\in\RR^3$ and 
$\vec{v}=(v_1,v_2,v_3)\in\RR^3$, 
let 
$$
\vec{u}\cdot\vec{v}=u_1v_1+u_2v_2+u_3v_3, 
\quad 
\lvert\vec{u}\rvert=\sqrt{\vec{u}\cdot\vec{u}}, 
$$
$$
\vec{u}\times\vec{v}
=
(u_2v_3-u_3v_2,u_3v_1-u_1v_3,u_1v_2-u_2v_1). 
$$
Let $\mathbb{S}^2$ be the two-sphere in $\RR^3$, that is, 
$\mathbb{S}^2
=
\{\vec{u}\in\RR^3 \ \vert \ \lvert\vec{u}\rvert=1\}$. 
In 1906, Da Rios formulated the equation modeling 
the motion of vortex filament of the form 
\begin{equation}
\vec{u}_t
= 
\vec{u} \times \vec{u}_{xx}, 
\label{equation:darios}
\end{equation}
where $\vec{u}(t,x)$ is $\mathbb{S}^2$-valued. 
See his celebrated paper \cite{DR}, and other references 
\cite{Hasimoto} and \cite{LL} for instance. 
The physical model \eqref{equation:darios} is 
an example of the equation of the one-dimensional Schr\"odinger map. 
Our equation \eqref{equation:pde} 
in the setting $b=a/2$ 
geometrically generalizes 
an $\mathbb{S}^2$-valued physical model 
\begin{equation}
\vec{u}_t
= 
\vec{u} \times \vec{u}_{xx}
+
a
\left[
\vec{u}_{xxx} 
+
\frac{3}{2}
\{\vec{u}_x\times (\vec{u}\times \vec{u}_x )\}_x
\right]
\label{equation:NT}
\end{equation}
proposed by Fukumoto and Miyazaki in \cite{FM}. 
\par
Here we state the known results on the mathematical analysis 
of the IVP \eqref{equation:pde}-\eqref{equation:data}. 
In case $a=b=0$, 
there are many studies on the existence theorem for 
\eqref{equation:pde}-\eqref{equation:data}. 
See 
\cite{CSU},
\cite{Ding} 
\cite{Koiso}, 
\cite{MCGAHAGAN},
\cite{NSU1}, 
\cite{NSU2}, 
\cite{PWW2}, 
\cite{SSB}
and references therein. 
In \cite{SSB} Sulem, Sulem and Bardos 
treated the higher dimensional ferromagnetic spin system of the form
\begin{equation}
\vec{u}_t 
=
\vec{u}\times \Delta_{\RR^{m}}\vec{u},
\label{equation:SSB}
\end{equation} 
where 
$\vec{u}(t,x)$ is the $\mathbb{S}^2$-valued function 
of $(t,x)\in \RR \times \RR^m$,
$\Delta_{\RR^{m}}$ is the Euclid Laplacian on $\RR^m$.
They proved global existence of smooth solution with small initial data, 
whereas they also proved that the problem admits the time-global 
solution with large data only if $m=1$. 
In \cite{Koiso} Koiso proved the local existence theorem 
of the IVP for the one-dimensional Schr\"odinger map 
for closed curves into K\"ahler manifolds of the form
\begin{equation}
u_t = J_u \nabla_x u_x
\label{equation:1dsf}
\end{equation}
in 
$H^{m+1}(\TT;N)$ for any integer $m\geqslant2$. 
Furthermore, he proved that the problem admits time-global solution 
if $N$ is a locally hermitian symmetric space. 
Recently, higher dimensional Schr\"odinger map into K\"ahler manifolds
has been studied. 
This equation is not only 
the higher dimensional version of \eqref{equation:1dsf}, 
but also the geometric generalization of \eqref{equation:SSB}. 
See e.g.
\cite{Ding}, \cite{MCGAHAGAN} and \cite{PWW2} 
for the detail. 
\par
In case $a\neq 0$, 
only $\mathbb{S}^2$-valued dispersive flow has been studied. 
In \cite{NT} Nishiyama and Tani proved the global existence theorem 
of the IVP for \eqref{equation:NT} 
in $H^{m+1}(\TT;\RR^3)$ with an integer $m\geqslant 2$ 
in the setting $b=a/2$. 
Moreover, 
they formulated the IVP for curves with two fixed edges on $\mathbb{S}^2$ 
at $x=\pm \infty$ for $x\in \RR$, 
and proved the global existence results also. 
\par
The purpose of this paper is to study the existence theorem of 
\eqref{equation:pde}-\eqref{equation:data} 
especially in the setting that $a\neq 0$, $b\in \RR$ and 
$N$ is a general K\"ahler manifold. 
To state our result, 
we now introduce some definitions related to Sobolev spaces for mappings. 
\begin{definition}
\label{definition:sobolev}
Let $(N,g)$ be a Riemannian manifold, 
and let $\mathbb{N}$ be the set of positive integers. 
For $m\in\mathbb{N}\cup\{0\}$, 
a bundle-valued Sobolev space of mappings is defined by  
$$
 H^{m+1}(\TT;N)
 =
 \{
u 
\ \vert \ 
u(x)\in N  
\ \text{a.e.} \ 
x \in \TT,
\ \text{and} \
u_x \in H^m(\TT;TN) 
\},
$$
where $u_x \in H^m(\TT;TN)$ means that $u_x$ satisfies 
$$
\lVert{u_x}\rVert_{H^{m}(\TT;TN)}^2
=
\sum_{j=0}^m
\int_0^1 
g_{u(x)}(\nabla_x^ju_x(x),\nabla_x^ju_x(x))dx
<
+\infty.
$$
Moreover, let $I$ be an interval in $\RR$, 
and let $w:(N,g)\to(\RR^d,g_0)$ be an isometric embedding. 
Here $g_0$ is the standard Euclidean metric on $\RR^d$.  
We say that $u\in C(I;H^{m+1}(\TT;N))$ 
if $u(t)\in H^{m+1}(\TT;N)$ for all $t\in{I}$ and 
$w{\circ}u\in C(I;H^{m+1}(\TT;\RR^d))$, 
where $C(I;H^{m+1}(\TT;\RR^d))$ is 
the set of usual Sobolev space valued continuous functions on $I$. 
\end{definition}
Our main results are the following.
\begin{theorem}
\label{theorem:eo}
Let $m\geqslant2$ be an integer. 
Then for any $u_0{\in}H^{m+1}(\TT;N)$, 
there exists a constant $T>0$ 
depending only on $a$, $b$, $N$ and 
$\lVert{u_{0x}}\rVert_{H^2(\TT;TN)}$ 
such that the initial value problem 
\eqref{equation:pde}-\eqref{equation:data} 
possesses a unique solution 
$u{\in}C([-T,T];H^{m+1}(\TT;N))$. 
\end{theorem}
Roughly speaking, Theorem~\ref{theorem:eo} says that 
\eqref{equation:pde}-\eqref{equation:data}
has a time-local solution in the usual Sobolev space $H^3$. 
In addition, 
$m=2$ is the smallest integer 
for \eqref{equation:pde} to make sense in the class
$C([-T,T];L^2(\TT;TN))$.
\par
We cannot prove any global existence results for 
\eqref{equation:pde}-\eqref{equation:data} 
independently of $a, b$ and $N$. 
In case $N=\mathbb{S}^2, a\neq 0$ and $b=a/2$, 
Nishiyama and Tani made use of some conservation laws 
to prove the global existence theorem in \cite{NT} and \cite{TN}. 
These conservation laws were discovered by Zakharov and Shabat 
in the study of the Hirota equation. 
See \cite{ZS} for the detail. 
If we take into account of the effect 
of the curvature of $N$ to the third term of \eqref{equation:pde}, 
we obtain the global existence theorem 
in the same way as \cite{NT} and \cite{TN} in case that the K\"ahler manifold
$N$ is a compact Riemann surface  
with constant Gaussian curvature as a $C^{\infty}$-manifold.  
We shall prove the following.
\begin{theorem}
\label{theorem:meo}
Let $(N,J,g)$ be a compact Riemann surface 
with constant Gaussian curvature $K$
and let $a\neq 0$ and $b=aK/2$. 
Then for any $u_0{\in}H^{m+1}(\TT;N)$ with an integer $m\geqslant2$, 
there exists a unique solution 
$u{\in}C(\RR;H^{m+1}(\TT;N))$
to the initial value problem 
\eqref{equation:pde}-\eqref{equation:data}.
\end{theorem}
We remark that Theorem~\ref{theorem:meo} 
generalizes the results of Nishiyama and Tani 
in \cite{NT} and \cite{TN}. 
In other words, 
the proof of Theorem~\ref{theorem:meo} will explain the reason 
why the global existence theorem of 
\eqref{equation:pde}-\eqref{equation:data} holds 
in case that $N=\mathbb{S}^2$. 
We would also like to recall that 
there are some classical examples of the 
compact Riemann surface
with constant Gaussian curvature.
Indeed, 
not only two-sphere $\mathbb{S}^2$ and 
flat-torus $\mathbb{T}^2=\RR^2/\mathbb{Z}^2$,
but also 
closed orientable surfaces $\varSigma_g$ with genus $g\geqslant 2$
admit the structure of such manifold.
The Gaussian curvature $K$ 
of them are $1$, $0$, and $-1$ respectively.
\par
Our method of the proof of Theorem~\ref{theorem:eo} is based on 
the geometric analysis and the classical energy method. 
We first remark that 
the local smoothing effect of dispersive equations breaks down 
because of the compactness of the domain $\TT$. 
See \cite{Doi} for the detail. 
Fortunately, however, 
\eqref{equation:pde} behaves like 
symmetric hyperbolic systems in some sense, 
and a geometric classical energy method works for \eqref{equation:pde}. 
More precisely, 
the K\"ahler condition $\nabla J \equiv 0$ 
and 
the properties of the Riemannian curvature tensor 
ensures that the loss of derivatives does not occur 
in geometric energy estimates. 
In other words, the solvable structure on the system of 
partial differential operators comes from 
the good geometric structures on $N$.  
In addition, 
we sometimes identify the unknown map $u$ with $w{\circ}u$ 
via the Nash isometric embedding 
$w:(N,J,g)\to(\RR^d,g_0)$ in our proof. 
It is more convenient to treat the system for $w{\circ}u$ 
than to treat \eqref{equation:pde} directly 
when we apply the standard argument of partial differential equations. 
\par
More concretely, 
the process of our proof of Theorem~\ref{theorem:eo} is as follows. 
We may assume that $N$ is compact  
since the initial curve $u_0$ lies on a compact subset in $N$ 
even if $N$ is noncompact.  
It suffices to solve the problem in the positive direction in time. 
First, we construct a sequence of approximate solutions 
$\{u^{\ep}\}_{\ep\in (0, 1)}$ to 
\begin{alignat}{2}
  u_t
& = 
  -
  \ep\,\nabla_x^3u_x
  +
  a\,\nabla_x^2u_x
  +
  J_u\nabla_xu_x
  +
  b\,g_u(u_x,u_x)u_x
&
  \quad\text{in}\quad
& (0,T_{\ep})\times{\TT},
\label{equation:pde2}
\\
  u(0,x)
& =
  u_0(x)
&
  \quad\text{in}\quad
& {\TT}.
\label{equation:data2}
\end{alignat}
By using a geometric orthogonal decomposition 
in the tubular neighbourhood of $w(N)$, 
we can check that a kind of maximum principle holds 
and $u^{\ep}(t)$ is $N$-valued. 
Secondly, 
the geometric classical energy estimates obtain 
the uniform estimate on the norm and the existence-time of
$\{u^{\ep }\}_{\ep\in (0, 1)}$. 
Then the standard compactness argument 
implies the local existence of solution 
$$
u{\in}C([0,T]\times\TT;N), 
\quad 
u_x
\in
C([0,T];H^{m-1}(\TT;TN)) 
\cap
L^\infty(0,T;H^m(\TT; TN))
$$
of \eqref{equation:pde}-\eqref{equation:data}, 
where $L^\infty$ is the usual Lebesgue space. 
Thirdly, we prove the uniqueness of solution
by the energy estimate in $H^1$ 
of the difference of two solutions with same initial data. 
We can choose a good moving frame of the normal bundle of 
$w(N)$ in $\RR^d$, 
and thus the classical energy method works 
for the difference of two solutions also. 
Finally, 
the continuity in time of 
$\nabla_x^mu_x$ in $L^2(\TT;TN)$ 
can be recovered by the standard argument. 
\par
The organization of this paper is as follows. 
Section \ref{section:note} is devoted to geometric preliminaries. 
In Section \ref{section:parabolic} 
we construct a sequence of approximate solutions 
by solving \eqref{equation:pde2}-\eqref{equation:data2}. 
In Section \ref{section:energy} 
we obtain uniform estimates of approximate solutions. 
In Section \ref{section:proof1} 
we complete the proof of Theorem \ref{theorem:eo}.
Finally, in Section \ref{section:global} 
we prove Theorem \ref{theorem:meo}. 
\section{Geometric Preliminaries}
\label{section:note}
In this section, 
we introduce notation, 
recall the relationship between the bundle-valued Sobolev spaces 
and the standard Sobolev spaces, 
and obtain the formulation of a system equivalent to 
\eqref{equation:pde} 
used later in our proof.
\par
We will use 
$C=C(\cdot, \ldots, \cdot)$ 
to denote a positive constant depending on 
the certain parameters, 
geometric properties of $N$, 
et al.
The partial differentiation is written by $\p$, 
or the subscript, e.g., 
$\p_xf$, $f_x$, 
to distinguish from the covariant derivative 
along the curve, e.g., $\nabla_x$.
Throughout this paper,   
$w$ is an isometric embedding mapping 
from $(N, J, g)$ into the standard Euclidean space $(\RR^d, g_0)$.
The existence of $w$ is ensured 
by the celebrated works of 
Nash \cite{Nash}, 
Gromov and Rohlin \cite{GR}, 
and related papers. 
\par
Let $u:\TT\to N$ be given. 
We denote $\Gamma(u^{-1}TN)$ 
by the space of sections of 
$u^{-1}TN$ over $\TT$.
For $V, W\in \Gamma(u^{-1}TN)$, 
define $L^2$-inner product of them by
$$
\int_{\TT}
g(V, W)
dx
=
\int_0^1
g_{u(x)}(V(x),W(x))
dx,
$$
and use the notation
$$
\|V\|_{L^2(\TT;TN)}^2
=
\int_{\TT}
g(V, V)
dx.
$$
Then the quantity
$\|u_x\|_{H^m(\TT;TN)}^2$ 
defined in Definition~\ref{definition:sobolev} 
is written as
$$
\|u_x\|_{H^m(\TT;TN)}^2
=
\sum_{j=0}^m
\|\nabla_x^ju_x\|_{L^2(\TT;TN)}^2.
$$
At this time we see that
$\|u_x\|_{H^m(\TT;TN)}<\infty$ 
if and only if 
$\|(w\circ u)_x\|_{H^m(\TT;\RR^d)}<\infty$. 
See, e.g., \cite[Section~1]{SZ} or \cite[Proposition~2.5]{MCGAHAGAN} 
on this equivalence.
Noting this equivalence and the compactness of $\TT$, 
we have
\begin{align}
  H^{m+1}(\TT;N)
& =
  \{ 
    \ u \ 
    | 
    \ u(x)\in N  
    \ \text{a.e.} \ x\in \TT,
    \ \text{and} \
    (w \circ u)_{x}\in H^m(\TT;\RR^d ) 
    \
  \}
\nonumber
\\  
& =
  \{ 
   \ u \ 
   | 
   \ u(x)\in N  
   \ \text{a.e.} \ x\in \TT,
   \ \text{and} \
   w \circ u \in H^{m+1}(\TT;\RR^d ) \
  \}.
\nonumber
\end{align}
We will make use of fundamental Sobolev space theory 
of $H^{m+1}(\TT;\RR^d)$ later in our proof.
\par
Set $I=[-T,T]$ for $T>0$. 
The equation \eqref{equation:pde} 
is equivalent to a system 
for $w\circ u$ as follows.
\begin{lemma}
\label{lemma:equiv}
Assume that
$m\geqslant 2$ 
is an integer. 
Then 
$u\in C(I;H^{m+1}(\TT;N))$ 
satisfies 
\eqref{equation:pde}-\eqref{equation:data} 
if and only if 
$v=w\circ u\in C(I;H^{m+1}(\TT;\RR^d))$ 
is $w(N)$-valued and satisfies
\begin{alignat}{2}
  v_{t}
 =& 
  a\{ 
     v_{xxx} 
     + [A(v)(v_x, v_x)]_x 
     + A(v)(v_{xx}+A(v)(v_x, v_x), v_x) 
   \}
\nonumber
\\
& + 
  \Tilde{J}_{v}(v_{xx}+A(v)(v_x, v_x))
  + b|v_x|^2 v_x
&
  \quad\mathrm{in}\quad
& I\times
  \TT,
\label{equation:pde3}
\\
  v(0,x)
 =&
  w \circ u_0(x)
&
  \quad\mathrm{in}\quad
& {\TT}.
\label{equation:data3}
\end{alignat}
Here,
$
A(v)(\cdot, \cdot ): 
T_vw(N)\times T_vw(N)\to (T_vw(N))^{\perp}
$
is the second fundamental form 
of $w(N)\subset \RR^d$ 
and 
$
\Tilde{J}_v
=
dw_{w^{-1}\circ v}J_{w^{-1}\circ v}dw^{-1}_v
$ 
on $T_vw(N)$ 
at $v\in w(N)$ respectively.
\end{lemma}
\begin{proof}[Proof of Lemma~\ref{lemma:equiv}]
Suppose 
$u\in C(I;H^{m+1}(\TT;N))$ 
satisfies 
\eqref{equation:pde}-\eqref{equation:data}. 
Since $m\geqslant 2$, 
the mapping
$v=w\circ u: I\times \mathbb{T}\to w(N)$ 
satisfies
$$
  v_t
 = (w\circ u)_t 
 = dw_u(u_t)
 = a\, dw_u(\nabla_x^2u_x)
   +dw_u(J_u\nabla_xu_x)
   +b\,dw_u(g_u(u_x, u_x)u_x)
$$
in $C(I;L^2(\TT;\RR^d))$. 
Moreover we deduce
\begin{alignat}{2}
 dw_u(\nabla_xu_x)
& =
 v_{xx}+A(v)(v_x, v_x),
\label{equation:20} 
\\
 dw_u(\nabla_x^2u_x)
& =
 [dw_u(\nabla_xu_x)]_x
 +
 A(v)(dw_u(\nabla_xu_x), v_x)
\nonumber 
\\
& =
 v_{xxx}+[A(v)(v_x,v_x)]_x
 +A(v)(v_{xx}+A(v)(v_x, v_x), v_x),
\label{equation:21}
\\
dw_u(g_u(u_x,u_x)u_x)
& =
 g(u_x,u_x)dw_u(u_x)
  =
 \lvert
 v_x
 \rvert^2
 v_x
\label{equation:22}
\end{alignat}
from 
the definition of the covariant derivative 
on $u^{-1}TN$
 and 
the isometricity of $w$. 
Combining 
\eqref{equation:20}, 
\eqref{equation:21} 
and 
\eqref{equation:22}, 
we obtain that $v$ solves 
\eqref{equation:pde3}-\eqref{equation:data3} 
 in the class 
$C(I; L^2(\TT;\RR^d))$.
\par
Conversely, suppose 
$v\in C(I;H^{m+1}(\TT;\RR^d))$ 
takes value in $w(N)$ and solves 
\eqref{equation:pde3}-\eqref{equation:data3}. 
Since $dw$ is injective, 
it immediately follows from the same calculus 
as above that 
$u=w^{-1}\circ v\in C(I;H^{m+1}(\TT;N))$ 
solves 
\eqref{equation:pde3}-\eqref{equation:data3} 
in 
$C(I;L^2(\TT;TN))$. 
\end{proof}
\section{Parabolic Regularization}
\label{section:parabolic}
Assume that $N$ is compact 
in this section.
The aim of this section is 
to obtain a sequence 
$\{u^{\ep}\}_{\ep\in (0, 1)}$ 
solving 
\begin{alignat}{2}
  u_t
& = 
  -
  \ep\, \nabla_x^3u_x 
  +
  a\, \nabla_x^2u_x
  +
  J_u\nabla_xu_x
  +
  b\, g_u(u_x, u_x)u_x
&
  \quad\text{in}\quad
& (0,T_{\ep})\times{\TT},
\label{equation:pde4}
\\
  u(0,x)
& =
  u_0(x)
&
  \quad\text{in}\quad
& {\TT}
\label{equation:data4}
\end{alignat}
for each $\ep\in (0, 1)$, 
where $u=u^{\ep}(t,x)$ 
is also an $N$-valued unknown function of 
$(t,x)\in [0, T_{\ep}]\times \TT$, 
and $u_0$ is the same initial data as that of 
\eqref{equation:pde}-\eqref{equation:data} 
independent of $\ep\in (0, 1)$.
\par
In the same way as Lemma~\ref{lemma:equiv}, 
\eqref{equation:pde4}-\eqref{equation:data4} 
is equivalent to  the following problem
\begin{alignat}{2}
  v_t
& = 
  -\ep v_{xxxx}
  +F(v)
&
  \quad\text{in}\quad
& (0,T_{\ep})\times{\TT},
\label{equation:pde5}
\\
  v(0,x)
& =
  w\circ u_0(x)
&
  \quad\text{in}\quad
& {\TT},
\label{equation:data5}
\end{alignat}
where 
$v=v^{\ep}(t,x)$ 
is a $w(N)$-valued unknown function of 
$(t,x)\in [0, T_{\ep}]\times \TT$. 
Here
\begin{align}
 F(v) 
 =
& -\ep
  \{ [A(v)(v_x, v_x)]_{xx}
     +[A(v)(v_{xx}+A(v)(v_x, v_x), v_x)]_x
\nonumber
\\
& \qquad 
    +A(v)(v_{xxx}+[A(v)(v_x,v_x)]_x
    + A(v)(v_{xx}+A(v)(v_x, v_x), v_x), v_x)
   \}
\nonumber
\\
& + a
  \{ v_{xxx} + [A(v)(v_x, v_x)]_x 
    + A(v)(v_{xx}+A(v)(v_x, v_x), v_x) 
  \}
\nonumber
\\
&
 + \Tilde{J}_v(v_{xx}+A(v)(v_x, v_x))
 + b|v_x|^2 v_x.
\nonumber 
\end{align}
For $F(v)$, notice that there exists 
$G\in C^{\infty}(\RR^{4d};\RR^d)$ 
such that
\begin{equation}
F(v)=G(v, v_x, v_{xx}, v_{xxx}),
\quad
G(v, 0, 0, 0)=0,
\nonumber
\end{equation}
for $v:\TT\to w(N)$. 
Note that
\eqref{equation:pde5} 
is a system of fourth-order parabolic equations 
for $w(N)$-valued function 
and 
represents the equality of sections of  $v^{-1}Tw(N)$.
We show the following.
%
%
\begin{proposition}
\label{proposition:pr}
Let $u_0\in H^{m+1}(\TT;N)$ with an integer $m\geqslant 2$. 
Then for each $\ep\in (0, 1)$, 
there exists a constant
$
T_{\ep}
=
T(\ep, a, b, N,
\| (w\circ u_0)_x   \|_{H^m(\TT;\RR^d)} )>0
$
such that 
\eqref{equation:pde5}-\eqref{equation:data5} 
has a unique solution 
$v=v^{\ep}\in C([0, T_{\ep}]; H^{m+1}(\TT;\RR^d))$ 
satisfying 
$v(t,x)\in w(N)$ 
for all 
$(t,x)\in [0,T_{\ep}]\times \TT$.
\end{proposition}
For the solution $v$ in Proposition~\ref{proposition:pr}, 
the equivalence between 
\eqref{equation:pde4}-\eqref{equation:data4} 
and 
\eqref{equation:pde5}-\eqref{equation:data5} 
implies 
that
$u=w^{-1}\circ v$ 
solves 
\eqref{equation:pde4}-\eqref{equation:data4}.
The proof of this proposition consists of 
the following two steps.
We first construct the solution of 
\eqref{equation:pde5}-\eqref{equation:data5} 
whose image are contained in a tubular neighbourhood 
of $w(N)$ in $\RR^d$. 
Namely, for $\delta>0$, 
let $(w(N))_{\delta}$ be a $\delta$-tubular neighbourhood 
of  $w(N)\subset \RR^d$ defined by 
$$
(w(N))_{\delta}=
\left\{
Q=(q, X)\in \RR^d \ | \ 
q\in w(N), \
X\in (T_{q}w(N))^{\perp}, \ 
|X|< \delta \
\right\},
$$ 
and let $\pi :(w(N))_{\delta}\to w(N)$ 
be the nearest point projection map defined by 
$\pi(Q)=q$ for $Q=(q,X)\in (w(N))_{\delta}$. 
Since $w(N)$ is compact, 
it is well-known that, 
for any sufficiently small $\delta$, 
$\pi$ exists and is smooth. 
We fix such small $\delta$, 
and 
construct a unique time-local solution of 
\eqref{equation:pde5}-\eqref{equation:data5} 
in the class
$$
Y_T^{m,\delta}
 =
  \{
   v\in 
   C([0, T];H^{m+1}(\TT;\RR^d) )
   \ | \
   \| v-w\circ u_0 \|_{L^{\infty}((0,T)\times \TT;\RR^d)}
   \leqslant 
   \delta/2
  \}
$$
for sufficiently small $T>0$. 
The second step is to check that this solution is actually $w(N)$-valued 
by using a kind of maximum principle. 
In short, it suffices to show the following two lemmas.
%
%
\begin{lemma}
\label{lemma:calculus}
For each $\varepsilon\in (0,1)$, 
there exists a constant
$
T_{\ep}>0
$
depending on 
$
\ep, a, b, N 
$
and
$
\| (w\circ u_0)_x   \|_{H^m(\TT;\RR^d)} 
$
and there exists a unique solution 
$v\in Y_{T_{\ep}}^{m,\delta}$
to
\begin{alignat}{2}
  v_t
& = 
  -\ep v_{xxxx}+F(\pi\circ v)
&
  \quad\mathrm{in}\quad
& (0,T_{\ep})\times{\TT},
\label{equation:pde6}
\\
  v(0,x)
& =
  w\circ u_0(x)
&
  \quad\mathrm{in}\quad
& {\TT}.
\label{equation:data6}
\end{alignat}
\end{lemma}
%
%
\begin{lemma}
\label{lemma:maximum}
Fix $\varepsilon\in (0,1)$. 
Assume that 
$v=v^{\ep}\in Y_{T_{\ep}}^{m,\delta}$ 
solves \eqref{equation:pde6}-\eqref{equation:data6}. 
Then $v(t,x)\in w(N)$ 
for all 
$(t,x)\in [0,T_{\ep}]\times \TT$, 
thus $v$ solves \eqref{equation:pde5}-\eqref{equation:data5}. 
\end{lemma}
\begin{proof}[Proof of Lemma~\ref{lemma:calculus}] 
The idea of the proof is due to the contraction mapping argument. 
\par
Let $L$ be a nonlinear map defined by 
$$
Lv(t)
=
S_{\ep}(t)v_0 
+
\int_0^tS_{\ep}(t-s)F(\pi\circ v)(s)ds,
$$
where 
$v_0=w\circ u_0$ 
and for $\psi\in H^{m+1}(\TT;\RR^d)$
$$
S_{\ep}(t)\psi(x)
=
\sum_{n=-\infty}^{\infty}
e^{-\ep t(2\pi n)^4 + 2\pi inx}
\int_0^1
e^{-2\pi iny}\psi (y)dy
$$
is the solution of the linear problem 
associated to 
\eqref{equation:pde6}-\eqref{equation:data6}.
Set $M=\|v_{0x}\|_{H^m(\TT;\RR^d)}$, 
and 
define the space
$$
Z_T^{m,\delta}
 =
  \{
   v\in 
   Y_T^{m,\delta} \ | \
   \| v_x  \|_{L^{\infty}(0, T; H^m(\TT;\RR^d))}
   \leqslant
   2M
  \},
$$
which is a closed subset of the Banach space 
$C([0, T]; H^{m+1}(\TT;\RR^d))$. 
To complete the proof,
we have only to show that the map $L$ has a unique fixed point 
in $Z_{T_{\ep}}^{m,\delta}$ for sufficiently small $T_{\ep}$, 
since the uniqueness in the whole space 
$Y_{T_{\ep}}^{m,\delta}$ follows by similar and standard argument.
\par
The operator $-\ep\p^4_x$ 
gains the regularity of order $3$, 
since 
$
\ep^{j/4}t^{j/4}|n|^j
e^{-\ep t(2\pi n)^4}
$
is bounded 
for $j=0, 1, 2, 3$. 
In fact, there exists $C_1>0$ 
such that for any 
$\psi \in H^{m-2}(\TT;\RR^d)$
\begin{equation}
\| S_{\ep}(t)\psi \|_{H^{m+1}(\TT;\RR^d)}
 \leqslant
  C_1 \ep^{-3/4}t^{-3/4}
  \|\psi \|_{H^{m-2}(\TT;\RR^d)}
\label{equation:smoothing}
\end{equation}
holds for all $t\in [0, T]$.
\par
On the other hand, 
if $v$ belongs to the class $Z_{T}^{m, \delta}$, 
we see $v(t)\in C(\TT; (w(N))_{\delta})$ 
and 
$\|v_x(t)\|_{H^m(\TT;\RR^d)}
\leqslant 
2M$
follows for all $t\in [0, T]$. 
Thus, noting the form of $F(v)$ and the compactness of $w(N)$, 
it is easy to check that there exists 
$C_2=C_2(a,b,M,N)>0$ such that 
\begin{align}
  \| F(\pi\circ v)\|_{H^{m-2}(\TT;\RR^d)}
& \leqslant
  C_2\| v_x \|_{H^m(\TT;\RR^d)},
\label{equation:composite1}
\\
  \| F(\pi\circ u)-F(\pi\circ v)\|_{H^{m-2}(\TT;\RR^d)}
& \leqslant
  C_2
  \| u_x-v_x \|_{H^m(\TT;\RR^d)}
\label{equation:composite2}
\end{align}
for any 
$u, v\in Z_T^{m, \delta}$. 
\par 
Using the smoothing property 
\eqref{equation:smoothing} 
and the nonlinear estimates 
\eqref{equation:composite1} 
and 
\eqref{equation:composite2}, 
we can prove $L$ is a contraction mapping 
from $Z_{T_{\ep}}^{m, \delta}$ into itself 
provided that $T_{\ep}$ is sufficiently small. 
It is the standard argument, 
thus we omit the detail.
\end{proof}
%
%
\begin{proof}[Proof of Lemma~\ref{lemma:maximum}] 
Suppose 
$v\in Y_{T_{\ep}}^{m,\delta}$ 
solves \eqref{equation:pde6}- \eqref{equation:data6}.
Define $\rho: (w(N))_{\delta}\to \RR^d$ by 
$\rho(Q)=Q-\pi(Q)$ for $Q\in (w(N))_{\delta}$. 
Then it follows from the definition that
$\left|
\rho \circ v(t,x)
\right|
=
\min_{Q'\in w(N)}
\left|
v(t,x)-Q'
\right|$
since $w(N)$ is compact.  
In addition, 
$(v(t)-w\circ u_0)$ 
belongs to 
$L^{\infty}(\TT;\RR^d)$ 
since 
$v\in Y_{T_{\ep}}^{m,\delta}$. 
Thus $\rho \circ v(t)$ 
makes sense in $L^2(\TT;\RR^d)$ for each $t$. 
To obtain that $v$ is $w(N)$-valued, we show 
$$
\| \rho \circ v(t)\|_{L^2(\TT;\RR^d)}^2
=
\left\langle
 \rho \circ v(t), 
  \rho \circ v(t) 
\right\rangle
=0
$$
for all $t\in[0,T_{\ep}]$.
Since $\pi + \rho$ is identity on $(w(N))_{\delta}$,
\begin{equation}
 d\pi_v + d\rho_v
 =
 I_d 
\label{equation:decomposition}
\end{equation}
follows on $T_v(w(N))_{\delta}$, 
where $I_d$ is the identity. 
By identifying 
$T_v(w(N))_{\delta}$ 
with 
$\RR^d$, 
we see that
$v_t(t,x)\in T_{v(t,x)}(w(N))_{\delta}$
and
$
d\pi_{v}(v_t)(t,x)
\in 
T_{\pi\circ v(t,x)}w(N)
$ 
for each $(t,x)$. 
Thus 
$
\left\langle
\rho \circ v, d\pi_v ( v_t) 
\right\rangle
=0
$
holds. Using this relation and \eqref{equation:decomposition}, we deduce
$$
\frac{1}{2}\frac{d}{dt}
 \| \rho \circ v  \|_{L^2(\TT;\RR^d)}^2
=
\left\langle
  \rho \circ v, d\rho_v ( v_t) 
 \right\rangle
=
 \left\langle
  \rho \circ v, d\rho_v ( v_t)+d\pi_v ( v_t)
 \right\rangle
=
 \left\langle
  \rho \circ v, v_t 
 \right\rangle.
$$
Here let us notice that 
$(-\ep (\pi\circ v)_{xxxx}+F(\pi\circ v))(t)
\in 
\Gamma((\pi\circ v)^{-1}Tw(N))$ 
since $\pi\circ v(t)\in w(N)$, 
and thus this is perpendicular to 
$\rho\circ v(t)$. 
Noting this and substituting \eqref{equation:pde6}, we get
\begin{alignat}{2}
 \frac{1}{2}
 \frac{d}{dt}
 \| \rho \circ v  \|_{L^2(\TT;\RR^d)}^2
& = 
 \left\langle
 \rho \circ v, 
  -\ep v_{xxxx}
  +F(\pi\circ v) 
 \right\rangle
\nonumber
\\ 
&=
 \left\langle
 \rho \circ v, 
  - \ep(\rho\circ v)_{xxxx}
  - \ep(\pi\circ v)_{xxxx}
  +F(\pi\circ v) 
 \right\rangle
\nonumber
\\
& =
 \left\langle
 \rho \circ v, 
 - \ep(\rho\circ v)_{xxxx} 
 \right\rangle
\nonumber \\
&=
  - \ep
 \| (\rho \circ v)_{xx} \|_{L^2(\TT;\RR^d)}^2
\leqslant 
0,
\nonumber
\end{alignat}
which implies 
$
 \| \rho \circ v(t)\|_{L^2(\TT;\RR^d)}^2
 \leqslant
 \| \rho \circ v_0  \|_{L^2(\TT;\RR^d)}^2
 =0.
$ 
Hence 
$\rho \circ v(t)=0$ 
holds for all $t$. 
Thus $v(t)$ is $w(N)$-valued for all $t$. 
This completes the proof.
\end{proof}
\section{Geometric and classical energy estimate}
\label{section:energy}
Assume that $N$ is compact also in this section. 
Let $\{u^{\ep}\}_{\ep\in (0, 1)}$ 
be a sequence of solutions to 
\eqref{equation:pde4}-\eqref{equation:data4} 
constructed in Section~\ref{section:parabolic}.  
We will evaluate the bundle-valued 
Sobolev norms of 
$\{u_x^{\ep}\}_{\ep\in (0, 1)}$ 
and obtain the uniform estimate on the norm 
and the existence time.  
Our goal of this section is the following.
\begin{lemma}
\label{lemma:energy} 
Let 
$u_0\in H^{m+1}(\TT;N)$ 
with an integer 
$m\geqslant 2$, 
and let $\{u^{\ep}\}_{\ep\in (0, 1)}$ 
be a sequence of solutions to 
\eqref{equation:pde4}-\eqref{equation:data4}.  
Then there exists a constant  
$T>0$ 
depending only on 
$a, b, N, \|u_{0x}\|_{H^2(\TT;TN)}$  
such that 
$\{u_x^{\ep}\}_{\ep\in (0, 1)}$ 
is bounded in 
$L^{\infty}(0, T; H^m(\TT;TN))$.
\end{lemma}
\begin{proof}[Proof of Lemma~\ref{lemma:energy}]
To obtain the desired uniform bounds, 
we show that
\begin{align}
& \frac{d}{dt}
  \| u_{x}^{\ep}(t) \|_{H^2(\TT;TN)}^2
  \leqslant
  C(a,b,N)\sum_{r=4}^8\| u_x^{\ep}(t) \|_{H^2(\TT;TN)}^r,
\label{equation:energy2}
\\
& \frac{d}{dt}
  \| u_x^{\ep}(t) \|_{H^k(\TT;TN)}^2
  \leqslant 
  C(a,b,N, 
  \| u_x^{\ep}(t) \|_{H^{k-1}(\TT;TN)}
  )
  \| u_x^{\ep}(t) \|_{H^k(\TT;TN)}^2, 
\quad 
3\leqslant k\leqslant m,
\label{equation:energyk}
\end{align}
hold for all $t\in [0,T_{\ep}]$.
\par
Throughout the proof of \eqref{equation:energy2} 
and \eqref{equation:energyk}, 
we simply write $u$, $J$, $g$ 
in place of 
$u^{\ep}$, $J_u$, $g_u$ respectively,  
$\|\cdot\|_{H^k}=\|\cdot\|_{H^k(\TT;TN)}$, 
$\|\cdot\|_{L^2}=\|\cdot\|_{L^2(\TT;TN)}$, 
$\|\cdot\|_{L^{\infty}}=\|\cdot\|_{L^{\infty}(\TT;TN)}$ 
for $k\in \mathbb{N}$,
and sometimes omit to write time variable $t$. 
\par
Let $2\leqslant k\leqslant m$. 
We consider the following quantity
\begin{align}
  \frac{1}{2}
  \frac{d}{dt}
  \| u_x \|_{H^k}^2
& =
  \frac{1}{2}
  \sum_{l=0}^k
  \frac{d}{dt}
  \| \nabla_x^lu_x \|_{L^2}^2
  =
  \sum_{l=0}^k
  \int_{\TT}
  g\left(
  \nabla_t\nabla_x^lu_x, \nabla_x^lu_x 
  \right)dx .
\label{equation:a0}
\end{align}
Note that
$\nabla_tu_x=\nabla_xu_t$ 
and 
$\nabla_t\nabla_xu_x=\nabla_x\nabla_tu_x+R(u_t,u_x)u_x$ 
follows from the definition of the covariant derivative, 
where $R$ denotes the curvature tensor on $(N, J, g)$. 
Using these commutative relations inductively, 
we have for $l\geqslant 1$
\begin{equation}
\nabla_t\nabla_x^lu_x
=
\nabla_x^{l+1}u_t
+
\sum_{j=0}^{l-1}
\nabla_x^j
\left[ 
R(u_t, u_x)\nabla_x^{l-(j+1)}u_x 
\right].
\label{equation:a1}
\end{equation}
Substituting \eqref{equation:a1} and \eqref{equation:pde4} 
into \eqref{equation:a0} 
gives
\begin{align}
&\frac{1}{2} 
  \frac{d}{dt}
  \| u_x \|_{H^k}^2
\nonumber \\
&=
  \sum_{l=0}^k
  \int_{\TT}
  g\left( 
  \nabla_x^{l+1}u_t, \nabla_x^lu_x
  \right)dx
 +
  \sum_{l=1}^k\sum_{j=0}^{l-1}
  \int_{\TT}
  g\left(
  \nabla_x^j
  \left[ 
  R(u_t, u_x)\nabla_x^{l-(j+1)}u_x 
  \right], 
  \nabla_x^lu_x 
  \right)dx 
\nonumber
\\
&= -\ep 
   \sum_{l=0}^k
   \int_{\TT}
   g\left(
   \nabla_x^{l+4}u_x, \nabla_x^{l}u_x
   \right)dx
\nonumber \\
&\quad
  +a\sum_{l=0}^k
   \int_{\TT}
   g\left(
   \nabla_x^{l+3}u_x, \nabla_x^lu_x
   \right)dx
\nonumber \\
&\quad  +\sum_{l=0}^k
   \int_{\TT}
   g\left(
   \nabla_x^{l+1}J\nabla_xu_x, \nabla_x^{l}u_x
   \right)dx
\nonumber \\
&\quad
  +b\sum_{l=0}^k
   \int_{\TT}
   g\left(
   \nabla_x^{l+1}[g(u_x, u_x)u_x], \nabla_x^lu_x
    \right)dx
\nonumber \\
&\quad  -\ep 
   \sum_{l=1}^k
   \sum_{j=0}^{l-1}
   \int_{\TT}
   g\left(
   \nabla_x^j
   \left[ 
   R(\nabla_x^3u_x, u_x)\nabla_x^{l-(j+1)}u_x 
   \right], 
   \nabla_x^lu_x 
   \right)dx
\nonumber \\
&\quad  +a\sum_{l=1}^k 
    \sum_{j=0}^{l-1}
    \int_{\TT}
    g\left( 
    \nabla_x^j
    \left[ 
    R(\nabla_x^2u_x, u_x)\nabla_x^{l-(j+1)}u_x 
    \right], 
   \nabla_x^lu_x 
   \right)dx
\nonumber \\
&\quad  +\sum_{l=1}^k
   \sum_{j=0}^{l-1}
   \int_{\TT}
   g\left(
   \nabla_x^j
   \left[ 
   R(J\nabla_xu_x, u_x)\nabla_x^{l-(j+1)}u_x 
   \right], 
   \nabla_x^lu_x
   \right)dx 
\nonumber \\
&\quad  +b\sum_{l=1}^k
   \sum_{j=0}^{l-1}
   \int_{\TT}
  g\left(
  \nabla_x^j
  \left[ 
  g(u_x, u_x)R(u_x, u_x)\nabla_x^{l-(j+1)}u_x 
  \right], 
  \nabla_x^lu_x
  \right)dx.
\label{equation:a2}
\end{align}
Note that the last term of \eqref{equation:a2} 
equals to $0$ 
since $R(u_x, u_x)=0$.  
We next deduce
\begin{align}
  \int_{\TT}
  g\left(
   \nabla_x^{l+4}u_x, \nabla_x^lu_x 
  \right)dx
&= \int_{\TT}
  g\left(
   \nabla_x^{l+2}u_x, \nabla_x^{l+2}u_x 
  \right)dx
\nonumber \\
&
= \|\nabla_x^{l+2}u_x \|_{L^2}^2,
\label{equation:a3}
\\
  \int_{\TT}
  g\left(
  \nabla_x^{l+3}u_x, \nabla_x^{l}u_x 
   \right)dx
&= -\int_{\TT}
  g\left(
  \nabla_x^{l+2}u_x, \nabla_x^{l+1}u_x 
  \right)dx
\nonumber \\
&= -\frac{1}{2}\int_{\TT}
 \left[
g\left(
 \nabla_x^{l+1}u_x, \nabla_x^{l+1}u_x
 \right)
\right]_x
 dx
\nonumber \\
&
=0,
\label{equation:a4}
\intertext{by integrating by parts. 
In addition, 
$\nabla_xJ=J\nabla_x$ 
follows from the K\"ahler condition. 
Thus, by using this relation and the antisymmetricity of $J$, 
we have}
  \int_{\TT} 
  g\left(
  \nabla_x^{l+1}J\nabla_xu_x, \nabla_x^lu_x
  \right)dx
&=- \int_{\TT} 
  g\left(
  J\nabla_x^{l+1}u_x, \nabla_x^{l+1}u_x  
  \right)dx
=0.
\label{equation:a5}
\end{align}
Substituting  
\eqref{equation:a3}, 
\eqref{equation:a4} and 
\eqref{equation:a5} 
into 
\eqref{equation:a2} 
yields
\begin{equation}
\label{equation:a6}
\frac{1}{2}\frac{d}{dt}\| u_x \|_{H^k}^2
+\ep\sum_{l=0}^k\|\nabla_x^{l+2}u_x \|_{L^2}^2
=\operatorname{I}_k
+\operatorname{I}\hspace{-.1em}\operatorname{I}_k
+\operatorname{I}\hspace{-.1em}\operatorname{I}\hspace{-.1em}\operatorname{I}_k
+\operatorname{I}\hspace{-.1em}\operatorname{V}_k,
\end{equation}
where
\begin{align}
 \operatorname{I}_k
& =b\sum_{l=0}^k
  \int_{\TT}
  g\left(
  \nabla_x^{l+1}[g(u_x, u_x)u_x], \nabla_x^lu_x
  \right)dx,
\nonumber \\
 \operatorname{I}\hspace{-.1em}\operatorname{I}_k
& =a\sum_{l=1}^k 
   \sum_{j=0}^{l-1}
   \int_{\TT}
   g\left( 
   \nabla_x^j
  \left[ 
   R(\nabla_x^2u_x, u_x)\nabla_x^{l-(j+1)}u_x 
   \right], 
   \nabla_x^lu_x
   \right)dx,
\nonumber \\
 \operatorname{I}\hspace{-.1em}\operatorname{I}\hspace{-.1em}\operatorname{I}_k
& =\sum_{l=1}^k
  \sum_{j=0}^{l-1}
  \int_{\TT}
   g\left(
   \nabla_x^j
   \left[ 
   R(J\nabla_xu_x, u_x)\nabla_x^{l-(j+1)}u_x 
   \right], 
   \nabla_x^lu_x
   \right)dx,
\nonumber \\
 \operatorname{I}\hspace{-.1em}\operatorname{V}_k
& =-\ep 
   \sum_{l=1}^k
   \sum_{j=0}^{l-1}
   \int_{\TT}
   g\left(
   \nabla_x^j
  \left[ 
   R(\nabla_x^3u_x, u_x)\nabla_x^{l-(j+1)}u_x 
   \right], 
   \nabla_x^lu_x 
   \right)dx.
\nonumber
\end{align}
We show the desired bounds of 
$\operatorname{I}_k$, 
$\operatorname{I}\hspace{-.1em}\operatorname{I}_k$, 
$\operatorname{I}\hspace{-.1em}\operatorname{I}\hspace{-.1em}\operatorname{I}_k$,
 and 
$\operatorname{I}\hspace{-.1em}\operatorname{V}_k$ 
 below. 
\vspace{1em}
\\
\textbf{Case~1: $k=2.$}\\
We first consider 
$\operatorname{I}_2$. 
Using  
H\"older's inequality 
and the Sobolev embedding,  
we deduce
\begin{align}
 \operatorname{I}_2
=&b\int_{\TT}
  g\left( 
  \nabla_x[g(u_x, u_x)u_x], u_x 
  \right)dx
  \nonumber \\
&+b\int_{\TT}
  g\left( 
  \nabla_x^2[g(u_x, u_x)u_x], \nabla_xu_x 
  \right)dx
  \nonumber \\
&+b\int_{\TT}
   g\left( 
   \nabla_x^3[g(u_x, u_x)u_x], \nabla_x^2u_x 
   \right)dx 
\nonumber \\
\leqslant&
 C(b)\|u_x\|_{H^2}^4
  +
 b\int_{\TT}
   g\left( 
   \nabla_x^3[g(u_x, u_x)u_x], \nabla_x^2u_x 
   \right)dx. 
\nonumber 
\end{align}
%
%
Furthermore a simple calculation gives
\begin{align}
\nabla_x^3[g(u_x, u_x)u_x]
=&
2g\left(\nabla_x^3u_x, u_x\right)u_x
+
g\left(u_x, u_x\right)\nabla_x^3u_x
\nonumber \\
&+
6g\left(\nabla_x^2u_x, \nabla_xu_x\right)u_x
+
6g\left(\nabla_x^2u_x, u_x\right)\nabla_xu_x
\nonumber \\
&+
6g\left(\nabla_xu_x, u_x\right)\nabla_x^2u_x
+
6g\left(\nabla_xu_x, \nabla_xu_x\right)\nabla_xu_x,
\nonumber
\end{align}
hence we deduce
\begin{align}
 \int_{\TT}
 g\left( 
   \nabla_x^3[g(u_x, u_x)u_x], \nabla_x^2u_x 
  \right)dx
=&
  2\int_{\TT}
   g\left( 
   g\left(\nabla_x^3u_x, u_x\right)u_x, 
   \nabla_x^2u_x 
   \right)dx
\nonumber \\
&
  +\int_{\TT}
  g\left( 
   g(u_x, u_x)\nabla_x^3u_x, 
   \nabla_x^2u_x 
   \right)dx
\nonumber \\
& 
  +12\int_{\TT}
  g\left( 
   g\left(\nabla_x^2u_x, u_x\right)\nabla_xu_x, 
   \nabla_x^2u_x 
   \right)dx
\nonumber \\
&
  +6\int_{\TT}
   g\left( 
   g\left(\nabla_xu_x, u_x\right)\nabla_x^2u_x, 
   \nabla_x^2u_x 
   \right)dx
\nonumber \\
&
  +6\int_{\TT}
  g\left( 
   g\left(\nabla_xu_x, \nabla_xu_x\right)\nabla_xu_x, 
   \nabla_x^2u_x 
   \right)dx.
\label{equation:b1}
\end{align} 
We see 
$\nabla_x^3u_x$ disappears 
from the first and the second term 
of the right hand side of \eqref{equation:b1} 
by a good symmetricity of $g$. 
In fact, 
after integrating by parts,  
we have  
\begin{align}
 2\int_{\TT}
  g\left( 
   g\left(\nabla_x^{3}u_x, u_x\right)u_x, 
   \nabla_x^2u_x 
   \right)dx
 & =-2\int_{\TT}
  g\left( 
   g\left(\nabla_x^2u_x, u_x\right)\nabla_xu_x, 
   \nabla_x^2u_x 
   \right)
   dx,
\label{equation:b2} \\
\int_{\TT}
  g\left( 
  g(u_x, u_x)\nabla_x^3u_x, 
   \nabla_x^2u_x 
   \right)
   dx 
 & = -\int_{\TT}
  g\left( 
   g\left(\nabla_xu_x, u_x\right)\nabla_x^2u_x, 
   \nabla_x^2u_x 
   \right)
   dx.
\label{equation:b3}
\intertext{
Substituting \eqref{equation:b2}, \eqref{equation:b3} 
into \eqref{equation:b1} 
and noting}
\int_{\TT}
g\left( 
g\left(\nabla_xu_x, \nabla_xu_x\right)\nabla_xu_x, 
\nabla_x^2u_x 
\right)
 dx
&=
\frac{1}{4}
\int_{\TT}
\left[
g\left(\nabla_xu_x,\nabla_xu_x\right)^2
\right]_x
dx
=0,
\nonumber 
\end{align}
we obtain  
\begin{align}
& \int_{\TT}
  g\left( 
   \nabla_x^3[g(u_x, u_x)u_x], \nabla_x^2u_x 
   \right)
   dx
\nonumber \\
&=
 10\int_{\TT}
  g\left( 
   g\left(\nabla_x^2u_x, u_x\right)\nabla_xu_x, 
   \nabla_x^2u_x 
   \right)
   dx
+5\int_{\TT}
 g\left( 
   g\left(\nabla_xu_x, u_x\right)\nabla_x^2u_x, 
   \nabla_x^2u_x 
   \right)
   dx,
\nonumber
\end{align}
which is bounded by $C\|u_x\|_{H^2}^4$.
Therefore we have
\begin{equation}
 \operatorname{I}_2
\leqslant 
C(b)\|u_x\|_{H^2}^4.
\label{equation:alpha}
\end{equation}
\par 
Next we consider 
$\operatorname{I}\hspace{-.1em}\operatorname{I}_2$. 
A simple computation gives  
\begin{align}
 \operatorname{I}\hspace{-.1em}\operatorname{I}_2
 =& a\int_{\TT}
    g\left( 
   \nabla_x
   \left[ R(\nabla_x^2u_x, u_x)u_x 
   \right], 
   \nabla_x^{2}u_x 
   \right)
   dx
\nonumber \\
&+a\int_{\TT}
  g\left( 
   R(\nabla_x^2u_x, u_x)\nabla_xu_x, 
   \nabla_x^2u_x 
   \right)
   dx
\nonumber \\
& +a\int_{\TT}
   g\left( 
   R(\nabla_x^2u_x, u_x)u_x, 
   \nabla_xu_x 
   \right)
   dx.
\label{equation:b4}
\end{align}
Moreover it follows from the definition of the covariant derivative 
of $R$ 
\begin{align} 
  \nabla_x
  \left[ R(\nabla_x^2u_x, u_x)u_x 
  \right]
=&
 (\nabla R)(u_x)(\nabla_x^2u_x, u_x)u_x 
+
 R(\nabla_x^3u_x, u_x)u_x
\nonumber \\
&+
  R(\nabla_x^2u_x, \nabla_xu_x)u_x
+
  R(\nabla_x^2u_x, u_x)\nabla_xu_x.
\label{equation:b5}
\end{align}
By noting that
\begin{equation}
g(R(X, Y)Z, W)=g(R(W, Z)Y, X)
\label{equation:b6}
\end{equation}
holds for any $X, Y, Z, W\in \Gamma(u^{-1}TN)$, 
and by substituting \eqref{equation:b5} 
into \eqref{equation:b4}, 
we deduce  
\begin{align}
\operatorname{I}\hspace{-.1em}\operatorname{I}_2
 =&a\int_{\TT}
   g\left( 
   (\nabla R)(u_x)(\nabla_x^2u_x, u_x)u_x, 
   \nabla_x^2u_x 
   \right)
   dx
\nonumber \\
& 
   +
   a\int_{\TT}
   g\left( 
   R(\nabla_x^3u_x, u_x)u_x, 
   \nabla_x^2u_x 
   \right)
   dx
\nonumber \\
& 
  +3a\int_{\TT}
   g\left( 
   R(\nabla_x^2u_x, \nabla_xu_x)u_x, 
   \nabla_x^2u_x 
   \right)
   dx
\nonumber \\
& 
  +a\int_{\TT}
   g\left( 
   R(\nabla_x^2u_x, u_x)u_x, 
   \nabla_xu_x 
   \right)
   dx
\nonumber \\
\leqslant& 
C(a, N)(\|u_x\|_{H^2}^4+\|u_x\|_{H^2}^5)
+  a\int_{\TT}
   g\left( 
   R(\nabla_x^3u_x, u_x)u_x, 
   \nabla_x^2u_x 
   \right)
   dx.
\label{equation:b7}
\end{align}
Here we used the fact that 
$R$ and $\nabla R$ are bounded operators 
since $N$ is compact. 
Furthermore, $\nabla_x^3u_x$ disappears from 
the second term of the right hand side of 
\eqref{equation:b7} 
because of the properties of $R$. 
In fact, we deduce from the integration by parts 
and \eqref{equation:b6}
\begin{align}
 a\int_{\TT}
  g\left( 
   R(\nabla_x^3u_x, u_x)u_x, 
   \nabla_x^2u_x 
   \right)
   dx
=& a\int_{\TT}
  g\left( 
   R(\nabla_x^2u_x, u_x)u_x, 
   \nabla_x^3u_x 
   \right)
   dx
\nonumber \\
=&
 \frac{a}{2}
 \int_{\TT}
  g\left( 
   R(\nabla_x^2u_x, u_x)u_x, 
   \nabla_x^3u_x 
   \right)
  dx
\nonumber \\
&
-
  \frac{a}{2}
   \int_{\TT}
   g\left( 
   R(\nabla_x^3u_x, u_x)u_x, 
   \nabla_x^2u_x 
   \right)
    dx
\nonumber \\
& 
   -\frac{a}{2}
   \int_{\TT}
   g\left( 
   R(\nabla_x^2u_x, \nabla_xu_x)u_x, 
   \nabla_x^2u_x 
   \right)
   dx
\nonumber \\
&
-
  \frac{a}{2}
   \int_{\TT}
   g\left( 
   R(\nabla_x^2u_x, u_x) \nabla_xu_x, 
   \nabla_x^2u_x 
   \right)
   dx
\nonumber \\
& 
  -\frac{a}{2}
   \int_{\TT}
   g\left( 
   (\nabla R)(u_x)(\nabla_x^2u_x, u_x)u_x, 
   \nabla_x^2u_x 
   \right)
   dx
\nonumber \\
=&-a\int_{\TT}
   g\left( 
   R(\nabla_x^2u_x, \nabla_xu_x)u_x, 
   \nabla_x^2u_x 
   \right)
   dx
\nonumber \\
& 
  -\frac{a}{2}
   \int_{\TT}
   g\left( 
   (\nabla R)(u_x)(\nabla_x^2u_x, u_x)u_x, 
   \nabla_x^2u_x 
   \right)
   dx, 
\label{equation:b8}
\end{align}
which is also bounded by 
$C(a, N)(\|u_x \|_{H^2}^4+\|u_x \|_{H^2}^5)$. 
Thus we get 
%
\begin{equation}
 \operatorname{I}\hspace{-.1em}\operatorname{I}_2
 \leqslant
C(a, N)(\|u_x \|_{H^2}^4+\|u_x \|_{H^2}^5).
\label{equation:beta}
\end{equation}
\par 
Next we compute 
$\operatorname{I}\hspace{-.1em}\operatorname{I}
\hspace{-.1em}\operatorname{I}_2$. 
This can be treated as the estimation of the composite function of 
lower order terms. Indeed, a simple computation gives
\begin{align}
 \operatorname{I}\hspace{-.1em}\operatorname{I}\hspace{-.1em}\operatorname{I}_2
 =&
  \int_{\TT}
  g\left( 
  \nabla_x
  \left[ R(J\nabla_xu_x, u_x)u_x \right], 
   \nabla_x^{2}u_x 
  \right)
  dx
\nonumber \\
&  +
  \int_{\TT}
  g\left( 
   R(J\nabla_xu_x, u_x)\nabla_xu_x, 
   \nabla_x^2u_x 
   \right)
   dx
\nonumber \\
&  +\int_{\TT}
   g\left( 
   R(J\nabla_xu_x, u_x)u_x, \nabla_xu_x 
   \right)
   dx
\nonumber \\
\leqslant&
 C(N)( \|u_x\|_{L^{\infty}}^3
       \|J\nabla_xu_x\|_{L^2}
       \|\nabla_x^2u_x\|_{L^2}
     +\|u_x\|_{L^{\infty}}^2
      \|J\nabla_x^2u_x\|_{L^2}
      \|\nabla_x^2u_x\|_{L^2}
\nonumber \\
& \qquad \quad 
    +\|u_x\|_{L^{\infty}} 
     \|\nabla_xu_x\|_{L^{\infty}}
     \|J\nabla_xu_x\|_{L^2} 
     \|\nabla_x^2u_x\|_{L^2}
\nonumber \\
& \qquad \quad 
    +\|u_x\|_{L^{\infty}}^2 
     \|J\nabla_xu_x\|_{L^2} 
     \|\nabla_xu_x\|_{L^2})
\nonumber \\
\leqslant& 
  C(N)(\|u_x\|_{H^2}^4 + \|u_x\|_{H^2}^5 ).
\label{equation:gamma}
\end{align}
\par 
Next we compute $\operatorname{I}\hspace{-.1em}\operatorname{V}_2$. 
Using \eqref{equation:b6} 
and the Cauchy inequality, 
we deduce for any $A>0$
\begin{align}
 \operatorname{I}\hspace{-.1em}\operatorname{V}_2
=& -\ep 
    \int_{\TT}
    g\left( 
     \nabla_x
     \left[ R(\nabla_x^3u_x, u_x)u_x \right], 
   \nabla_x^2u_x 
     \right)
     dx
\nonumber \\
&
 -\ep
  \int_{\TT}
   g\left(
   R(\nabla_x^3u_x, u_x)\nabla_xu_x , 
   \nabla_x^2u_x 
   \right)
   dx
\nonumber \\
&-\ep 
  \int_{\TT}
   g\left( 
   R(\nabla_x^3u_x, u_x)u_x, 
   \nabla_xu_x 
   \right)
   dx
\nonumber \\
 =&
  -\ep 
  \int_{\TT}
   g\left( 
   (\nabla R)(u_x)
   (\nabla_x^2u_x, u_x)u_x, 
   \nabla_x^3u_x 
   \right)
   dx
\nonumber \\
 & -\ep 
  \int_{\TT}
   g\left( 
   R(\nabla_x^2u_x, u_x)u_x, 
   \nabla_x^4u_x 
   \right)
   dx
\nonumber \\
 &
 -\ep 
  \int_{\TT}
   g\left( 
   R(\nabla_x^2u_x, u_x)\nabla_xu_x, 
   \nabla_x^3u_x 
   \right)
   dx
\nonumber \\
&-2\ep 
  \int_{\TT}
   g\left( 
   R(\nabla_x^2u_x, \nabla_xu_x)u_x, 
   \nabla_x^3u_x 
   \right)
   dx
\nonumber \\
 &
 -\ep 
  \int_{\TT}
   g\left( 
   R(\nabla_xu_x, u_x)u_x, 
   \nabla_x^3u_x 
   \right)
    dx.
\nonumber \\
\leqslant
& \ep A 
  \| \nabla_x^4u_x\|_{L^2}^2
 +5\ep A
  \| \nabla_x^3u_x\|_{L^2}^2
\nonumber \\
& +\frac{\ep}{4A}
  \Bigl\{ 
   \| (\nabla R)(u_x)(\nabla_x^2u_x, u_x)u_x  \|_{L^2}^2
   +
   \| R(\nabla_x^2u_x, u_x)u_x  \|_{L^2}^2
\nonumber \\
& \phantom{+\frac{\ep}{4A}}
  +\| R(\nabla_x^2u_x, u_x)\nabla_xu_x  \|_{L^2}^2
  +2\| R(\nabla_x^2u_x, \nabla_xu_x)u_x  \|_{L^2}^2
  +\| R(\nabla_xu_x, u_x)u_x  \|_{L^2}^2
  \Bigr\} 
\nonumber \\
 \leqslant
  & \ep A
   \| \nabla_x^4u_x\|_{L^2}^2
   +
   5\ep A
   \| \nabla_x^3u_x\|_{L^2}^2
 +
   \frac{C(N)}{4A}
 ( \|u_x\|_{H^2}^6+\|u_x\|_{H^2}^8 ).
\label{equation:sigma}
\end{align}
Combining the estimates 
\eqref{equation:alpha}, 
\eqref{equation:beta}, 
\eqref{equation:gamma} 
and
\eqref{equation:sigma}
yields 
\begin{align}
& \frac{1}{2}
  \frac{d}{dt}
  \| u_x \|_{H^2}^2
+
 (1-A)\ep 
      \|\nabla_x^4u_x \|_{L^2}^2
+
 (1-5A)\ep 
      \|\nabla_x^3u_x \|_{L^2}^2
+
 \ep 
  \|\nabla_x^2u_x \|_{L^2}^2
\nonumber \\
& \phantom{=}
 \leqslant 
 C(a, b, N, A)(
        \|u_x\|_{H^2}^4+\|u_x\|_{H^2}^5
       + \|u_x\|_{H^2}^6+\|u_x\|_{H^2}^8
              ).
\nonumber
\end{align}
Especially take $A>0$ as $A< 1/5$, 
then we obtain the desired inequality \eqref{equation:energy2}.
\vspace{1em}
\\
\textbf{Case~2: $3\leqslant k\leqslant m.$}\\
Let $3\leqslant k\leqslant m$. We also compute
$ 
\operatorname{I}_k
+\operatorname{I}\hspace{-.1em}\operatorname{I}_k
+\operatorname{I}\hspace{-.1em}\operatorname{I}\hspace{-.1em}\operatorname{I}_k
+\operatorname{I}\hspace{-.1em}\operatorname{V}_k 
$ 
in \eqref{equation:a6}. 
We can obtain the desired inequality \eqref{equation:energyk} 
in the similar argument as in the case $k=2$. 
\par
We first consider $\operatorname{I}_k$. 
A simple computation gives
\begin{align}
 \operatorname{I}_k
 =&2b\sum_{l=0}^k
  \int_{\TT}
  g\left( 
   g\left(\nabla_x^{l+1}u_x, u_x\right)u_x, 
   \nabla_x^lu_x 
   \right)
   dx
\nonumber \\
 &+b\sum_{l=0}^k
   \int_{\TT}
   g\left( 
   g(u_x, u_x)\nabla_x^{l+1}u_x, 
   \nabla_x^lu_x 
   \right)
   dx
\nonumber \\
&+2b\sum_{l=0}^k
   (l+1)\int_{\TT}
   g\left( 
   g\left(\nabla_x^lu_x, \nabla_xu_x\right)u_x, 
   \nabla_x^lu_x 
   \right)
    dx
\nonumber \\
&+2b\sum_{l=0}^k
   (l+1)\int_{\TT}
   g\left( 
   g\left(\nabla_x^lu_x, u_x\right)\nabla_xu_x, 
   \nabla_x^lu_x 
   \right)
   dx
\nonumber \\
&+2b\sum_{l=0}^k
   (l+1)\int_{\TT}
   g\left( 
   g\left(\nabla_xu_x, u_x\right)\nabla_x^lu_x, 
   \nabla_x^lu_x 
   \right)
   dx
\nonumber \\
&
+ 
  P_k,
\intertext{where}
 P_k
 =&
 b\sum_{l=0}^k
       \sum_{\begin{smallmatrix}
        \alpha+\beta+\gamma=l+1 \\
        \alpha, \beta, \gamma\geqslant 0\\
        \max\{\alpha, \beta, \gamma\}\leqslant l-1
       \end{smallmatrix}}
  \frac{(l+1)!}{\alpha!\beta!\gamma!}
  \int_{\TT}
  g\left( 
  g\left(
  \nabla_x^{\alpha}u_x, \nabla_x^{\beta}u_x
  \right)
  \nabla_x^{\gamma}u_x, 
  \nabla_x^lu_x 
   \right)
   dx.
\nonumber
\end{align}
It is easy to check 
$P_k$ 
is bounded by  
$C(b)\|u_x\|_{H^{k-1}}^2 
\|u_x\|_{H^k}^2$.
On the other hand,  
$\operatorname{I}_k-P_k$ can be treated in the same way as 
in the case $k=2$ by using the estimation like \eqref{equation:b2}  
and 
\eqref{equation:b3}.   
Indeed, by integrating by parts 
and by applying the good structure of $g$, we have
\begin{align}
 2\int_{\TT}
  g\left( 
   g\left(\nabla_x^{l+1}u_x, u_x\right)u_x, 
   \nabla_x^lu_x 
   \right)
   dx
& =
  -2\int_{\TT}
   g\left( 
   g\left(\nabla_x^lu_x, u_x\right)\nabla_xu_x, 
   \nabla_x^lu_x 
   \right)
   dx,
\nonumber
\\
 \int_{\TT}
  g\left( 
   g(u_x, u_x)\nabla_x^{l+1}u_x, 
   \nabla_x^lu_x 
   \right)
   dx
&=
-\int_{\TT}
 g\left( 
   g\left(\nabla_xu_x, u_x\right)\nabla_x^lu_x, 
   \nabla_x^lu_x 
   \right)
   dx.
\nonumber
\end{align}
Therefore we deduce
\begin{align}
 \operatorname{I}_k-P_k 
=&
b\sum_{l=0}^k
   (4l+2)\int_{\TT}
   g\left( 
   g\left(\nabla_x^lu_x, \nabla_xu_x\right)u_x, 
   \nabla_x^lu_x 
   \right)
   dx
\nonumber \\
& +b\sum_{l=0}^k
  (2l+1)\int_{\TT}
   g\left( 
   g\left(\nabla_xu_x, u_x\right)\nabla_x^lu_x, 
   \nabla_x^lu_x 
   \right)
   dx
\nonumber \\
 \leqslant& 
 C|b|\sum_{l=0}^k
 \|u_x\|_{L^{\infty}}
 \|\nabla_xu_x\|_{L^{\infty}}
 \|\nabla_x^lu_x\|_{L^2}^2
\nonumber \\ 
\leqslant& 
 C(b)\|u_x\|_{H^2}^2
     \|u_x\|_{H^k}^2.
\nonumber
\end{align}
Consequently, we obtain the desired boundness 
$$\operatorname{I}_k 
\leqslant 
C(b, \|u_x\|_{H^{k-1}})\|u_x\|_{H^k}^2.$$
\par 
We next estimate  
$\operatorname{I}\hspace{-.1em}\operatorname{I}_k$. 
A simple computation yields 
\begin{align}
\operatorname{I}\hspace{-.1em}\operatorname{I}_k
=&a\sum_{l=1}^k
   \int_{\TT}
   g\left( 
   R(\nabla_x^{l+1}u_x, u_x)u_x, 
  \nabla_x^lu_x 
   \right)
   dx 
\nonumber \\
&
+a\sum_{l=1}^k
  (l-1)\int_{\TT}
  g\left( 
  (\nabla R)(u_x)
  (\nabla_x^lu_x, u_x)u_x, 
  \nabla_x^lu_x 
   \right)
   dx 
\nonumber \\
&
 +a\sum_{l=1}^k
  (2l-1)\int_{\TT}
  g\left( 
  R(\nabla_x^lu_x, \nabla_xu_x)u_x, 
  \nabla_x^lu_x 
   \right)
   dx
\nonumber \\
& 
+
Q_k,    
\nonumber
\end{align}
where 
\begin{align}
Q_k
&=a\sum_{l=1}^k
 \sum_{j=0}^{l-1}
 \sum_{\begin{smallmatrix}
        p+q+r+s=j \\
        p, q, r, s\geqslant 0\\ 
        \max\{p, q+2, r, s+l-(j+1)\}
       \leqslant l-1 
       \end{smallmatrix}}
A_{p,q,r,s}^j
\nonumber \\
&\qquad \qquad \qquad \qquad
 \times
 \int_{\TT}
 g\left( 
 (\nabla_x^pR)
  (\nabla_x^{q+2}u_x, \nabla_x^{r}u_x)\nabla_x^{s+l-(j+1)}u_x, 
  \nabla_x^lu_x 
   \right)
   dx,
\nonumber \\
\nabla_x^pR
&=
\sum_{\alpha =1}^p
\sum_{\begin{smallmatrix}
        \alpha + \sum_{h=1}^{\alpha}p_h=p \\
        p_h\geqslant 0
       \end{smallmatrix}}
B_{p_1,\ldots,p_{\alpha}}^{\alpha}
(\nabla^{\alpha} R)(\nabla_x^{p_1}u_x, \ldots, \nabla_x^{p_\alpha}u_x )
\nonumber 
\end{align}
for some constant 
$A_{p,q,r,s}^j$, 
$B_{p_1,\ldots,p_{\alpha}}^{\alpha}$ 
if $p\in \mathbb{N}$, 
and $\nabla_x^0R=R$. 
\par
On the estimation of 
$\operatorname{I}\hspace{-.1em}\operatorname{I}_k-Q_k$, 
the property of Riemannian curvature tensor 
works well similarly to the estimate \eqref{equation:beta}. 
Indeed, after integrating by parts, we have
\begin{align}
&\int_{\TT}
   g\left( 
   R(\nabla_x^{l+1}u_x, u_x)u_x, 
  \nabla_x^lu_x 
   \right)
   dx 
\nonumber \\
&
=
-\frac{1}{2}
\int_{\TT}
   g\left( 
   (\nabla R)(u_x)(\nabla_x^lu_x, u_x)u_x, 
  \nabla_x^lu_x 
   \right)
   dx
\nonumber \\
& \phantom{=}
-\int_{\TT}
   g\left( 
   R(\nabla_x^lu_x, u_x)u_x, 
  \nabla_x^lu_x 
   \right)
   dx,
\nonumber  
\end{align}
thus we deduce
\begin{align}
\operatorname{I}\hspace{-.1em}\operatorname{I}_k-Q_k
=&
a\sum_{l=1}^k
  (l-3/2) \int_{\TT}
  g\left( 
  (\nabla R)(u_x)
  (\nabla_x^lu_x, u_x)u_x, 
  \nabla_x^lu_x 
   \right)
   dx 
\nonumber \\
&
+a\sum_{l=1}^k
   (2l-2)\int_{\TT}
   g\left( 
   R(\nabla_x^lu_x, \nabla_xu_x)u_x, 
  \nabla_x^lu_x 
   \right)
   dx 
\nonumber \\
\leqslant&
C(a, N)
(\|u_x \|_{L^{\infty}}^3+\|u_x \|_{L^{\infty}}
\|\nabla_xu_x \|_{L^{\infty}})
\|u_x \|_{H^k}^2.
\label{equation:c1}
\end{align}
\par 
On the other hand, 
on the estimation of $Q_k$,  
if the integers
$p, q, r, s \geqslant 0$ satisfy 
$p+q+r+s=j$ and 
$\max\{p, q+2, r, s+l-(j+1)\} \leqslant l-1$, 
we can easily check that 
there are at most two elements of the 
set 
$\{p, q+2, r, s+l-(j+1)\}$ 
which equals to $l-1$, 
and that the others are not greater than $l-2$. 
Thus we deduce
$$
Q_k
 \leqslant
C(a)\sum_{l=1}^k
\left( 
\sum_{p=0}^{l-1}
\| \nabla_x^pR \|_{L^{\infty}}
\right)
\|u_x\|_{H^{l-1}}^2
\|u_x\|_{H^l}^2.
$$
Here let us notice  from definition that 
there may appear 
$u_x, \ldots, \nabla_x^{p-1}u_x$
in $\nabla_x^pR$, 
but there does not appear  $\nabla_x^pu_x$ 
in $\nabla_x^pR$. 
Noting this, it is easy to check
$$
\sum_{p=0}^{l-1}
\| \nabla_x^pR \|_{L^{\infty}}
\leqslant 
C(N)
\sum_{p=0}^{l-1}
\sum_{r=0}^p
\| u_x\|_{H^p}^r
\leqslant 
C(N)
\sum_{r=0}^{l-1}
\| u_x\|_{H^{l-1}}^r.
$$
Therefore we have
\begin{equation}
Q_k
\leqslant 
C(a, N)
\left(
\sum_{r=2}^{k+1}
\| u_x\|_{H^{k-1}}^r
\right)
\| u_x\|_{H^k}^2.
\label{equation:c2}
\end{equation}
Thus \eqref{equation:c1}  and \eqref{equation:c2} 
imply the desired boundness
$$
\operatorname{I}\hspace{-.1em}\operatorname{I}_k
\leqslant 
C(a, N, \| u_x\|_{H^{k-1}})
\| u_x\|_{H^k}^2.
$$
\par 
The desired boundness of  
$\operatorname{I}\hspace{-.1em}\operatorname{I}\hspace{-.1em}\operatorname{I}_k$ 
and 
$\operatorname{I}\hspace{-.1em}\operatorname{V}_k$ 
also follows from the same argument as that of 
$\operatorname{I}\hspace{-.1em}\operatorname{I}\hspace{-.1em}\operatorname{I}_2$ 
and 
$\operatorname{I}\hspace{-.1em}\operatorname{V}_2$. 
Indeed, we can easily deduce 
\begin{align}
\operatorname{I}\hspace{-.1em}\operatorname{I}\hspace{-.1em}\operatorname{I}_k
& \leqslant
C(N, \| u_x\|_{H^{k-1}})   
\| u_x\|_{H^k}^2, 
\nonumber
\intertext{and there exists $C_1>0$ such that for any $A>0$}
\operatorname{I}\hspace{-.1em}\operatorname{V}_k
& \leqslant
C_1\ep 
A\sum_{l=0}^k
\| \nabla_x^{l+2}u_x\|_{L^2}^2
+ 
C(N, A, \|u_x\|_{H^{k-1}} )
\|u_x\|_{H^k}^2.
\nonumber
\end{align}
\par
Applying these estimation of 
$\operatorname{I}_k$, 
$\operatorname{I}\hspace{-.1em}\operatorname{I}_k$, 
$\operatorname{I}\hspace{-.1em}\operatorname{I}
\hspace{-.1em}\operatorname{I}_k$, 
$\operatorname{I}\hspace{-.1em}\operatorname{V}_k$ 
to the right hand side of  \eqref{equation:a6} 
leads to
\begin{align} 
  \frac{1}{2}
  \frac{d}{dt}
 \|u_x\|_{H^k}^2
+(1-C_1A)\ep
 \sum_{l=0}^{k}
 \|\nabla_x^{l+2}u_x\|_{L^2}^2
\leqslant
C(a, b, N, A,\|u_x\|_{H^{k-1}} )
\|u_x\|_{H^k}^2.
\label{equation:c3}
\end{align}
Thus, by taking $A<1/C_1$, 
we obtain the desired inequality \eqref{equation:energyk}.
\par
By using 
\eqref{equation:energy2} and \eqref{equation:energyk}, 
we now complete the proof of Lemma~\ref{lemma:energy}. 
Set   
$f(t)=\|u_x^{\ep}(t)\|_{H^2}^2+1$, 
then we have from \eqref{equation:energy2},
\begin{equation}
\label{equation:u1} 
\frac{df}{dt}\leqslant C(a, b, N)f^4, \quad f(0)=\|u_{0x}\|_{H^2}^2+1.
\end{equation}
It follows from \eqref{equation:u1} that  
there exists a positive constant
$T=T(a, b, N,\|u_{0x}\|_{H^2})>0$ 
and a positive constant  
$C_2=C_2(a, b, N,\|u_{0x}\|_{H^2})>0$ 
such that 
\begin{equation}
\label{equation:u2}
\|u_x^{\ep}(t)\|_{H^2}
\leqslant 
C_2
\end{equation}
holds for all $t\in [0, T]$.
Furthermore, since  \eqref{equation:energyk} holds 
for $k=3$, 
\eqref{equation:u2} and the Gronwall inequality implies
$$
\|u_x^{\ep}(t)\|_{H^3}^2
\leqslant 
\|u_{0x}\|_{H^3}^2
\exp 
( C(a, b, N, C_2 )T),
$$
which implies the existence of a constant 
$C_3=C_3(a, b, N,\|u_{0x}\|_{H^3},T)>0$ 
such that 
$$
\|u_x^{\ep}(t)\|_{H^3}
\leqslant 
C_3
$$ 
holds for all $t\in [0,T]$. 
It is now clear that we can show, 
by using \eqref{equation:energyk} inductively 
for each $3\leqslant k\leqslant m$, 
the existence of a constant 
$C_m=C_m(a, b, N,\|u_{0x}\|_{H^m},T)>0$ 
such that 
$$
\sup_{t\in [0,T]}\|u_x^{\ep}(t)\|_{H^m}
\leqslant 
C_m.
$$
\par 
It is easy to find that the solution $u^{\ep}$ 
to \eqref{equation:pde4}-\eqref{equation:data4} 
with $\ep \in (0,1)$ 
must exists on the interval $[0, T]$. 
Otherwise we extend the time interval of existence to cover  
$[0,T]$, 
that is, we have  
$T_{\ep}\geqslant T$. 
Thus the lemma has been proved. 
\end{proof}
\begin{remark}
\label{remark:remark3}
$\{u_x^{\ep}\}_{\ep \in (0,1)}$ 
gains the regularity in the following sense.
That is, by applying \eqref{equation:c3} with $k=m$, 
and by integrating on $[0, T]$,  
we obtain
$$
2(1-C_1A)\ep
\sum_{l=0}^{m}
\|\nabla_x^{l+2}u_x^{\ep}\|_{L^2([0,T]\times \TT)}^2
\leqslant 
C(a, b, N, A, \|u_{0x}\|_{H^m})T
+
\|u_{0x}\|_{H^m}^2.
$$
This implies
$
\{\ep^{1/2}
\nabla_x^mu_x^{\ep}\}_{\ep \in (0,1)}
$
is bounded in 
$L^2(0,T;H^2(\TT;TN))$.
This property will be used 
in the compactness argument 
in the next section.
\end{remark}
\section{Proof of Theorem \ref{theorem:eo}}
\label{section:proof1}
We are now in a position to complete the proof of 
Theorem~\ref{theorem:eo}.
\begin{proof}[Proof of Theorem~\ref{theorem:eo}]
At first assume that $N$ is compact.
\begin{proof}[Proof of existence]
Suppose that $u_0\in H^{m+1}(\TT;N)$  
with the integer $m\geqslant 2$ is given. 
By
Proposition~\ref{proposition:pr} 
there exists a sequence  
$\{u^{\ep}\}_{\ep \in (0,1)}$ 
solving \eqref{equation:pde4}-\eqref{equation:data4} 
for each $\ep>0$. 
Moreover, Lemma~\ref{lemma:energy} implies 
there exists  
$T=T(a, b, N,\|u_{0x}\|_{H^2(\TT;TN)})>0$  
which is independent of $\ep \in (0,1)$  
such that 
$\{u^{\ep}_x\}_{\ep \in (0,1)}$ 
is bounded in 
$L^{\infty}(0,T;H^m(\TT;TN))$. 
Thus, since $\TT$ is compact, we have 
$\{v^{\ep}\}_{\ep \in (0,1)}$
is bounded in 
$L^{\infty}(0,T;H^{m+1}(\TT;\RR^d))$, 
where $v^{\ep}=w\circ u^{\ep}$.
On the other hand, as stated in 
Remark~\ref{remark:remark3} 
in the  previous section,
$
\{\ep^{1/2}
\nabla_x^mu_x^{\ep}\}_{\ep \in (0,1)}
$
is bounded in 
$L^2(0,T;H^2(\TT;TN))$. 
Noting this, we see 
$\{u^{\ep}_t\}_{\ep \in (0,1)}$ 
is bounded in 
$L^2(0,T;H^{m-2}(\TT;TN))$,  
which implies  
$\{v^{\ep}\}_{\ep \in (0,1)} $
is bounded in 
$C^{0,1/2}([0,T];H^{m-2}(\TT;\RR^d))$. 
Consequently, by interpolating the spaces
$L^{\infty}(0,T;H^{m+1}(\TT;\RR^d))$ 
and 
$C^{0,1/2}([0,T];H^{m-2}(\TT;\RR^d))$, 
we obtain that 
$\{v^{\ep}\}_{\ep \in (0,1)} $ 
is bounded in the class
$C^{0,\alpha}([0,T];H^{m+1-6\alpha}(\TT;\RR^d))$
for any $0< \alpha \leqslant 1/2$. 
Hence  we see from Rellich's theorem and the Ascoli-Arzela theorem 
that there exists a subsequence  
$\{v^{\ep(j)}\}_{j=1}^{\infty} $ 
and 
$$
v\in L^{\infty}(0,T;H^{m+1}(\TT;\RR^d))
\cap
C([0, T]; H^m(\TT; \RR^d))
$$
such that
\begin{alignat}{4}
& v^{\ep(j)}
\stackrel{w^{\star}}{\longrightarrow}
v
\quad
&
\text{in}
\quad
&
L^{\infty}(0,T;H^{m+1}(\TT;\RR^d))
\quad
&
\text{as}
\quad
&
j\to \infty,
\label{equation:converge1}
\\
& v^{\ep(j)}
\longrightarrow
v
\quad
&
\text{in}
\quad
&
C([0,T];H^m(\TT;\RR^d))
\quad
&
\text{as}
\quad
&
j\to \infty.
\label{equation:converge2}
\end{alignat}
In particular, 
we see from \eqref{equation:converge2} that  
$v\in C([0,T]\times \TT;w(N))$. 
Furthermore it is easy to check that $v$ is a solution of 
\eqref{equation:pde3}-\eqref{equation:data3} 
with the initial data $w\circ u_0$.  
Thus Lemma~\ref{lemma:equiv} implies that  
$u=w^{-1}\circ v\in C([0,T]\times \TT;N)$ 
satisfies
$$
u\in
L^{\infty}(0,T;H^{m+1}(\TT;N))
\cap
C([0,T];H^m(\TT;N))
$$
and solves \eqref{equation:pde}-\eqref{equation:data} 
with the initial data $u_0$, 
which completes the proof of the existence. 
\end{proof}
\begin{proof}[Proof of uniqueness]
Let $u,v\in L^{\infty}(0,T;H^{m+1}(\TT;N)) 
\cap C([0,T];H^{m}(\TT;N))$ 
be solutions of \eqref{equation:pde}-\eqref{equation:data} 
such that $u(0,x)=v(0,x)$. 
Identify $u,v$ with $w\circ u, w\circ v$. 
Then $u$ and $v$ satisfy 
\eqref{equation:pde3}-\eqref{equation:data3} 
with $u(0,x)=v(0,x)$, 
and $z=u-v$ makes sense as $\RR^d$-valued function.  
Taking the difference between two equations, 
we have 
$$
z_t-az_{xxx}=f(u, u_x, u_{xx})-f(v, v_x, v_{xx}), 
$$
where 
\begin{align}
f(u, u_x, u_{xx})
=&
a\left\{
\left[A(u)(u_x,u_x)\right]_x
+
A(u)(u_{xx}+A(u)(u_x,u_x), u_x)
\right\}
\nonumber \\
&+
\tilde{J}_u(u_{xx}+A(u)(u_x,u_x))
+
b\left|u_x\right|^2u_x.
\nonumber
\end{align}
To prove that $z=0$, we show that there exists 
a constant $C>0$  
depending only on  
$a, b, N$, and the quantities $\|u_x\|_{L^{\infty}(0,T;H^2(\TT;\RR^d))}$, 
$\|v_x\|_{L^{\infty}(0,T;H^2(\TT;\RR^d))}$
such that
$$
\frac{d}{dt}
\|z(t)\|_{H^1(\TT;\RR^d)}^2
\leqslant
C\|z(t)\|_{H^1(\TT;\RR^d)}^2.
$$
We write $C$ 
without commenting 
the dependence of the constant,  
simply write 
$\|\cdot\|_{H^1}=\|\cdot\|_{H^1(\TT;\RR^d)}$, 
$\|\cdot\|_{L^2}^2=\|\cdot\|_{L^2(\TT;\RR^d)}^2
=\left\langle \cdot, \cdot \right\rangle$
and omit to write time variable $t$ below.
\par
At first, since the mean value theorem shows that 
$$
f(u,u_x,u_{xx})-f(v,v_x,v_{xx})
=
O(|z|+|z_x|+|z_{xx}|),
$$
we can easily check 
$$
\frac{1}{2}
\frac{d}{dt}
\|z\|_{L^2}^2
=
\left\langle z, z_t \right\rangle
\leqslant 
C\|z\|_{H^1}^2
$$
by using the integration by parts.
Thus  
we concentrate on the estimate of
$$
\frac{1}{2}
\frac{d}{dt}
\|z_x\|_{L^2}^2
=
\left\langle 
z_x, z_{xt} 
\right\rangle
=
-\left\langle 
z_{xx}, z_t 
\right\rangle
=
-\left\langle 
z_{xx}, f_a+f_J+f_b 
\right\rangle,
$$
where
\begin{align}
f_a
=&
a\Bigl\{u_{xxx}+\left[A(u)(u_x,u_x)\right]_x
        +A(u)(u_{xx}+A(u)(u_x,u_x), u_x)
\nonumber \\      
       &-v_{xxx}-\left[A(v)(v_x,v_x)\right]_x
        -A(v)(v_{xx}+A(v)(v_x,v_x), v_x)
 \Bigr\},
\nonumber \\
f_J
=&
\tilde{J}_u(u_{xx}+A(u)(u_x,u_x))
      -\tilde{J}_v(v_{xx}+A(v)(v_x,v_x)),
\nonumber \\
f_b
=&
b\left(\left|u_x\right|^2u_x-\left|v_x\right|^2v_x
\right).
\nonumber
\end{align}
For any $y\in w(N)$, let 
$p(y)=d\pi_y:\RR^d \to T_yw(N)$ 
be the orthogonal projection onto 
the tangent space of $w(N)$ at $y$, 
and define $n(y)=I_d-p(y)$, 
where $I_d$ is the identity on $\RR^d$. 
Note that $p(y)$ and $n(y)$ 
behaves as symmetric matrix on $\RR^d$ 
respectively.
\par
On the estimation of 
$\left\langle 
z_{xx}, f_J 
\right\rangle$,
let us notice at first 
$$
\tilde{J}_v(v_{xx}+A(v)(v_x,v_x))
=
\tilde{J}_vp(v)v_{xx}.
$$
Since $\tilde{J}_vp(v):\RR^d\to \RR^d$ is antisymmetric,  
we obtain the desired boundness.
Indeed, 
$$
\left\langle 
z_{xx}, f_J 
\right\rangle
=
\left\langle 
z_{xx}, (\tilde{J}_up(u)-\tilde{J}_vp(v))u_{xx}  
\right\rangle
+
\left\langle 
z_{xx}, \tilde{J}_vp(v)z_{xx}  
\right\rangle,
$$ 
where the second term of the right hand side vanishes and 
the first term of the right hand side
is bounded by $C\|z\|_{H^1}^2$ by using the integration by parts 
and the mean value theorem. 
\par 
The desired boundness of
$\left\langle 
z_{xx}, f_b 
\right\rangle$
follows
from the facts that 
$\|v_x\|^2I_d:\RR^d\to \RR^d$
and
$(v_x, \cdot)v_x:\RR^d\to \RR^d$ 
are symmetric respectively
and $v_x$ is in $L^{\infty}(0,T;H^2(\TT;\RR^d))$.
It is not so difficult, hence we omit the detail. 
\par
Thus it suffices to consider 
$\left\langle 
z_{xx}, f_a 
\right\rangle$.
From the definition 
of the covariant derivative along the curve 
and the relations
$p(u)^2=p(u)$, 
$p(u)=I_d-n(u)$, 
we deduce
\begin{align}
&u_{xxx}
+
\left[A(u)(u_x,u_x)\right]_x
+
A(u)(u_{xx}+A(u)(u_x,u_x), u_x)
\nonumber \\
&=
p(u)\left[p(u)u_{xx}\right]_x
\nonumber \\
&=
p(u)u_{xxx}+
p(u)\left[p(u)\right]_xu_{xx}
\nonumber \\
&=
u_{xxx}-n(u)u_{xxx}
+p(u)\left[p(u)\right]_xu_{xx}.
\nonumber
\end{align}
Roughly speaking,  
$n(u)$ gains the regularity of order $1$
since $u$ is $w(N)$-valued. 
In fact, as is shown below,
$-n(u)u_{xxx}
+p(u)\left[p(u)\right]_xu_{xx}$
essentially behaves as lower order term 
and  does not cause 
any bad effects on the $H^1$-energy estimate.
We first decompose by  
$$
\left\langle 
z_{xx}, f_a 
\right\rangle
=
a(A_0+A_1+A_2+A_3),
$$
where
\begin{align}
A_0
&=
\left\langle 
z_{xx},z_{xxx} 
\right\rangle
=0, 
\nonumber \\
A_1
&=
-
\left\langle 
z_{xx}, (n(u)-n(v))u_{xxx} 
\right\rangle
+
\left\langle 
z_{xx}, p(v)\left[p(u)-p(v)\right]_xu_{xx} 
\right\rangle, 
\nonumber \\
A_2
&=
-
\left\langle 
z_{xx}, n(v)z_{xxx} 
\right\rangle
+
\left\langle 
z_{xx}, p(v)\left[p(v)\right]_xz_{xx} 
\right\rangle, 
\nonumber \\
A_3
&=
\left\langle 
z_{xx}, (p(u)-p(v))\left[p(u)\right]_xu_{xx} 
\right\rangle.
\nonumber 
\end{align}
Obviously $A_3$ is bounded by $C\|z\|_{H^1}^2$ 
by using the integration by parts and the mean value theorem. 
In addition, since $p(v)$ is symmetric and $p(v)^2=p(v)$ on $\RR^d$, 
we deduce 
\begin{align}
A_2
&=
-
\left\langle 
z_{xx}, n(v)z_{xxx} 
\right\rangle
+
\left\langle 
p(v)z_{xx},\left[p(v)\right]_xz_{xx} 
\right\rangle
\nonumber \\
&=
-
\left\langle 
z_{xx}, n(v)z_{xxx} 
\right\rangle
-
\left\langle 
z_{xx},p(v)z_{xxx} 
\right\rangle 
\nonumber \\
&=
-
\left\langle 
z_{xx}, z_{xxx} 
\right\rangle 
\nonumber \\
&=0.
\nonumber
\end{align}
We need to estimate $A_1$ carefully. 
At first, assume that there exists 
real-valued functions $G^j$ 
defined on a neighbourhood of $w(N)$ in $\RR^d$ 
satisfying $\operatorname{grad}G^j\neq 0$ 
for each $j=n+1, \ldots, d$ 
such that 
$$
w(N)
=
\left\{
\ v
\ |
\
G^{n+1}(v)= \cdots= G^d(v)=0
\
\right\}.
$$
In this case $n\in \mathbb{N}$ 
is the real-dimension of $w(N)$ as the 
compact submanifold of $\RR^d$. 
Note that there exists a smooth orthonormal frame  
$\{\nu^{n+1}, \cdots, \nu^d\}$ 
for the normal bundle $(Tw(N))^{\perp}$ 
globally on $w(N)$. 
In this setting we start the estimation of $A_1$. 
It follows from the properties of 
$p(v), n(v), p(u)$ and $n(u)$ that 
\begin{align}
A_1
&=
-
\left\langle 
z_{xx}, (n(u)-n(v))u_{xxx} 
\right\rangle
-
\left\langle 
z_{xx}, p(v)\left[n(u)-n(v)\right]_xu_{xx} 
\right\rangle 
\nonumber \\
&=
-
\left\langle 
z_{xx}, n(v)(n(u)-n(v))u_{xxx} 
\right\rangle
-
\left\langle 
z_{xx}, p(v)\left[(n(u)-n(v))u_{xx}\right]_x
\right\rangle 
\nonumber \\
&=
-
\left\langle 
n(v)z_{xx}, (n(u)-n(v))u_{xxx} 
\right\rangle
-
\left\langle 
p(v)z_{xx}, \left[(n(u)-n(v))u_{xx}\right]_x
\right\rangle. 
\label{equation:unique0} 
\end{align}
On the first term of \eqref{equation:unique0}, 
it is important to note 
\begin{equation}
n(v)z_{xx}
=
\sum_{j=n+1}^d
\left(
z_{xx}, \nu^j(v)
\right)
\nu^j(v)
=
O(|z_x|)
\label{equation:unique1}
\end{equation}
holds since $v$ is $w(N)$-valued. 
Indeed, 
by taking the derivative of 
$(v_x, \nu^j(v))=0$ 
with respect to $x$, 
we have $(v_{xx}, \nu^j(v))=-(v_x, [\nu^j(v)]_x)$ 
and thus a simple computation implies 
\begin{equation}
\bigl(z_{xx},\nu^j(v)\bigr)
=
-\bigl(z_x,\left[\nu^j(v)\right]_x
 \bigr)
-\bigl(u_x,\left[\nu^j(u)-\nu^j(v)\right]_x
 \bigr)
-\bigl(u_{xx},\nu^j(u)-\nu^j(v)
 \bigr),
\label{equation:unique2}
\end{equation}
which is $O(|z_x|)$.
Hence, by noting \eqref{equation:unique1}, we have 
$$
-
\left\langle 
n(v)z_{xx}, (n(u)-n(v))u_{xxx} 
\right\rangle
\leqslant 
C\|z_x\|_{L^2}\|z\|_{L^{\infty}}\|u_{xxx}\|_{L^2} 
\leqslant 
C\|z\|_{H^1}^2.
$$
On the second term of \eqref{equation:unique0}, 
we deduce 
\begin{align}
-
\left\langle 
p(v)z_{xx}, \left[(n(u)-n(v))u_{xx}\right]_x
\right\rangle
=&
-
\left\langle 
(p(v)-p(u))z_{xx}, \left[(n(u)-n(v))u_{xx}\right]_x
\right\rangle
\nonumber \\
&+
\left\langle 
n(u)z_{xx}, \left[(n(u)-n(v))u_{xx}\right]_x
\right\rangle
\nonumber \\
&
-
\left\langle 
z_{xx}, \left[(n(u)-n(v))u_{xx}\right]_x
\right\rangle.
\label{equation:unique3}
\end{align}
The first term of \eqref{equation:unique3} 
is obviously bounded by $C\|z\|_{H^1}^2$. 
The second term of \eqref{equation:unique3} 
is also bounded by $C\|z\|_{H^1}^2$ since 
$n(u)z_{xx}=O(|z_x|)$. 
We consider the third term of \eqref{equation:unique3}. 
We have  
\begin{align}
&(n(u)-n(v))u_{xx}
\nonumber \\
&=
\sum_{j=n+1}^d
(u_{xx}, \nu^j(u))\nu^j(u)
-
\sum_{j=n+1}^d
(u_{xx}, \nu^j(v))\nu^j(v)
\nonumber \\
&=
\sum_{j=n+1}^d
(u_{xx}, \nu^j(u))(\nu^j(u)-\nu^j(v))
+
\sum_{j=n+1}^d
(u_{xx}, \nu^j(u)-\nu^j(v))\nu^j(v) 
\nonumber \\
&=
-\sum_{j=n+1}^d
(u_{x}, \left[\nu^j(u)\right]_x)(\nu^j(u)-\nu^j(v))
+
\sum_{j=n+1}^d
(u_{xx}, \nu^j(u)-\nu^j(v))\nu^j(v). 
\nonumber 
\end{align}
Thus it follows that
\begin{align}
&\left[(n(u)-n(v))u_{xx}\right]_x
\nonumber \\
&=
-\sum_{j=n+1}^d
(u_{x}, \left[\nu^j(u)\right]_x)\left[\nu^j(u)-\nu^j(v)\right]_x
+
\sum_{j=n+1}^d
(u_{xx}, \left[\nu^j(u)-\nu^j(v)\right]_x)\nu^j(v)
\nonumber \\
&\quad +\sum_{j=n+1}^d
(u_{xxx}, \nu^j(u)-\nu^j(v))\nu^j(v)+
O(|z|). 
\nonumber
\end{align}
Moreover, noting that  
$(z_{xx}, \nu^j(v))=O(|z_x|)$ 
follows from \eqref{equation:unique2}, we get 
\begin{align}
\sum_{j=n+1}^d
\left\langle 
z_{xx}, (u_{xx}, \left[\nu^j(u)-\nu^j(v)\right]_x)\nu^j(v)
\right\rangle
&\leqslant 
C\|z\|_{H^1}^2,
\nonumber \\
\sum_{j=n+1}^d
\left\langle 
z_{xx}, (u_{xxx}, \nu^j(u)-\nu^j(v))\nu^j(v)
\right\rangle
&\leqslant 
C\|z\|_{H^1}^2.
\nonumber
\end{align}
Thus we have only to estimate the following quantity
\begin{equation}
\sum_{j=n+1}^d
\left\langle 
z_{xx}, (u_{x}, \left[\nu^j(u)\right]_x)\left[\nu^j(u)-\nu^j(v)\right]_x
\right\rangle.
\label{equation:unique4}
\end{equation}
Here we write 
$$
\left[\nu^j(u)\right]_x
=
D^j(u) u_x, 
\quad 
j=n+1, \ldots, d,
$$
where 
$D^j(u)=\operatorname{grad}\nu^j(u)$ is a 
$\RR^d\times \RR^d$-valued function.
Using this notation, we have 
\begin{align}
\eqref{equation:unique4}
=&
\sum_{j=n+1}^d
\left\langle 
z_{xx}, (u_{x}, \left[\nu^j(u)\right]_x)(D^j(u)-D^j(v))u_x
\right\rangle
\nonumber
\\
&+
\sum_{j=n+1}^d
\left\langle 
z_{xx}, (u_{x}, \left[\nu^j(u)\right]_x)D^j(v)z_x
\right\rangle.
\label{equation:unique5}
\end{align}
The first term of \eqref{equation:unique5} 
is obviously bounded by $C\|z\|_{H^1}^2$.
On the second term of \eqref{equation:unique5}, 
note first that the following relation
\begin{equation}
\nu^j(v)
=
\frac{\operatorname{grad}G^j(v)}
     {\lvert\operatorname{grad}G^j(v) \rvert}
=
\operatorname{grad}
\left(
\frac{G^j(v)}
{\lvert\operatorname{grad}G^j(v) \rvert}
\right)
\label{equation:unique6}
\end{equation}
holds at $v\in w(N)$. Thus it follows that 
$$
D^j(v)
=
\left(
\frac{\partial^2}{\partial v^{\alpha}\partial v^{\beta}}
\left(
\frac{G^j(v)}
{\lvert\operatorname{grad}G^j(v) \rvert}
\right)
\right)_{1\leqslant \alpha, \beta \leqslant d},
$$
which is a symmetric matrix valued. 
Then we deduce 
\begin{align}
\sum_{j=n+1}^d
\left\langle 
z_{xx}, (u_{x}, \left[\nu^j(u)\right]_x)D^j(v)z_x
\right\rangle
=-\frac{1}{2}
\sum_{j=n+1}^d
\left\langle 
z_x, \left[(u_{x}, \left[\nu^j(u)\right]_x)D^j(v)\right]_xz_x
\right\rangle
\nonumber 
\end{align}
which is bounded by $C\|z\|_{H^1}^2$.
Consequently we obtain the desired boundness of  
$A_1$.
\par
In the general case, 
there may not exists any 
global orthonormal frame  
for the normal bundle $(Tw(N))^{\perp}$ on $w(N)$. 
However, we can assume without loss of generality 
that 
$$
w(N)
=
\bigcup_{I=1}^L
\Omega_I
=
\bigcup_{I=1}^L
\left\{
\
v
\
|
\
G_I^{n+1}(v)
=
\cdots
=
G_I^{d}(v)
=
0
\
\right\}
$$
for some $L\in \mathbb{N}$ and real-valued
functions
$G_I^{n+1},
\ldots,
G_I^{d}$ 
defined in the neighbourhood of 
$\Omega_I$ in $\RR^d$ with 
$\operatorname{grad}G_I^j\neq 0$
for each $j$, $1\leqslant I\leqslant L$. 
Let $\{\lambda^I\}_{I=1}^L$ be a partition of unity 
associated to 
$\{\Omega_I\}_{I=1}^L$. 
Then on each $\Omega_I$, 
there exists a smooth 
orthonormal frame  
for the normal bundle 
satisfiying the relation like \eqref{equation:unique6}.
Furthermore, we can proceed almost the same argument 
as above by noting
$n(u)=n(u)\sum_{I=1}^L\lambda^I(u)$ and  
$[n(u)]_x=[n(u)]_x\sum_{I=1}^L\lambda^I(u)$.
It is not difficult, thus we omit the detail. 
\par
Consequently, we obtain the desired inequality 
$$
\frac{d}{dt}
\|z(t)\|_{H^1}^2
\leqslant
C
\|z(t)\|_{H^1}^2.
$$
Thus, since $z(0)=0$, Gronwall's inequality implies 
$z=0$.  
This completes the proof of the uniqueness. 
\end{proof}
\begin{proof}[Proof of the continuity in time of \ $\nabla_x^mu_x$
in $L^2(\TT;TN)$]
So far in our proof, 
we have proved the existence of a unique solution 
$u\in L^{\infty}(0,T;H^{m+1}(\TT;N)) 
\cap 
C([0,T];H^m(\TT;N)$. 
Let $v=w\circ u$. 
To obtain that
$u\in C([0,T];H^{m+1}(\TT;N)$, 
we show  
$v_x\in C([0,T];H^m(\TT;\RR^d))$.
Note that it follows from the definition of the covariant derivative
that 
\begin{equation}
dw_u(\nabla_x^mu_x)
=
\partial_x^{m+1}v
+
\sum_{l=2}^{m+1}
\sum_{\begin{smallmatrix}
      \alpha_1+\cdots +\alpha_l=m+1 \\
      \alpha_i\geqslant 1\\
      \end{smallmatrix}}
B_{(\alpha_1, \cdots, \alpha_l)}
(v)
(\partial_x^{\alpha_1}v_x, 
\cdots, 
\partial_x^{\alpha_l}v_x).
\label{equation:es0}
\end{equation}
Here
$
B_{(\alpha_1, \cdots, \alpha_l)}(\cdot)
$
are multi-linear vector-valued functions  
on $\RR^d$, 
and it is easy to check that the second term of 
the right hand side of \eqref{equation:es0} is in 
$C([0, T];L^2(\TT;\RR^d))$.
Thus it suffices to show that
$dw_u(\nabla_x^mu_x)$ 
belongs to  
$C([0,T];L^2(\TT;\RR^d))$. 
\par 
First of all, 
we can derive  
from the energy estimate \eqref{equation:energyk} 
and the isometricity of $w$ 
\begin{align}
\frac{d}{dt}
\|dw_{u^{\ep}}(\nabla_x^mu_x^{\ep})(t)\|_{L^2(\TT;\RR^d)}^2
&=
\frac{d}{dt}
\|\nabla_x^mu_x^{\ep}(t)\|_{L^2(\TT;TN)}^2
\leqslant 
C
\nonumber 
\intertext{for some $C>0$ which is independent of $\ep\in (0,1)$. 
Therefore it follows that} 
\|dw_{u^{\ep}}(\nabla_x^mu_x^{\ep})(t)\|_{L^2(\TT;\RR^d)}^2
&\leqslant
\|dw_u(\nabla_x^mu_{x})(0)\|_{L^2(\TT;\RR^d)}^2
+Ct.
\nonumber \\
\intertext{Letting $\ep \downarrow 0$, we have
$dw_u(\nabla_x^mu_x)(t)\in L^2(\TT;\RR^d)$ 
makes sense for all $t\in [0, T]$, and } 
\|dw_u(\nabla_x^mu_x)(t)\|_{L^2(\TT;\RR^d)}^2
&\leqslant
\|dw_u(\nabla_x^mu_{x})(0)\|_{L^2(\TT;\RR^d)}^2
+Ct,
\nonumber
\intertext{which leads to} 
\limsup_{t\to 0}
\|dw_u(\nabla_x^mu_x)(t)\|_{L^2(\TT;\RR^d)}^2
&\leqslant
\|dw_u(\nabla_x^mu_{x})(0)\|_{L^2(\TT;\RR^d)}^2.
\label{equation:es1}
\intertext{Moreover, 
since $v\in L^{\infty}(0,T;H^{m+1}(\TT;\RR^d)) 
\cap 
C([0,T];H^m(\TT;\RR^d))$, 
we see 
$dw_u(\nabla_x^mu_x)(t)$ 
is weakly continuous in $L^2(\TT;\RR^d)$. 
Hence it follows that}
\|dw_u(\nabla_x^mu_{x})(0)\|_{L^2(\TT;\RR^d)}^2
&\leqslant
\liminf_{t\to 0}
\|dw_u(\nabla_x^mu_x)(t)\|_{L^2(\TT;\RR^d)}^2.
\label{equation:es2}
\intertext{From \eqref{equation:es1} and \eqref{equation:es2}, 
we obtain} 
\lim_{t\to 0}
\|dw_u(\nabla_x^mu_x)(t)\|_{L^2(\TT;\RR^d)}^2
&=
\|dw_u(\nabla_x^mu_x)(0)\|_{L^2(\TT;\RR^d)}^2.
\label{equation:es3}
\end{align}
Consequently, it follows from \eqref{equation:es3} 
and the weak continuity of $dw_u(\nabla_x^mu_x)(t)$ 
in $L^2(\TT;\RR^d)$, 
$dw_u(\nabla_x^mu_x)(t)$ 
is strongly continuous in $L^2(\TT;\RR^d)$ 
at $t=0$. 
By the uniqueness of $u$,  
we see 
$dw_u(\nabla_x^mu_x)(t)$ 
is strongly continuous at each $t\in [0, T]$ in the same way. 
\end{proof}
\par 
Finally assume that $N$ is noncompact. 
In this case, 
retake  
$N'$ as a compact subset of $N$ in which  the image of initial 
data is contained. 
Then we can proceed the same argument on $N'$
as in the case $N$ is compact. 
Thus we complete the proof of Theorem~\ref{theorem:eo}. 
\end{proof}
\section{Global Existence}
\label{section:global}
The goal of this section is to prove Theorem \ref{theorem:meo}. 
Let $(N, J, g)$ be a compact Riemann surface  
with constant Gaussian curvature $K$, 
and assume that $a\neq 0$ and $b=aK/2$. 
Theorem~\ref{theorem:eo} tells us that, 
given a initial data 
$u_0\in H^{m+1}(\TT;N)$, 
there exists  
$T=T(a, b, N,\|u_{0x}\|_{H^2(\TT;TN)})>0$ 
such that the IVP 
\eqref{equation:pde}-\eqref{equation:data} 
admits a unique time-local solution  
$u\in C([0,T);H^{m+1}(\TT;N))$.
\par 
In what follows  
we will extend the existence time of $u$ 
over $[0,\infty)$. 
For this,  
we have the following energy conversation laws.
\begin{lemma}
\label{lemma:conserved}
For $u\in C([0,T);H^{m+1}(\TT;N))$ 
solving \eqref{equation:pde}-\eqref{equation:data}, 
the following quantities
\begin{align}
&\|u_x(t)\|_{L^2(\TT;TN)}^2,
\nonumber \\
&E(u(t))
=
\|\nabla_x^2u_x(t)\|_{L^2(\TT;TN)}^2
+
\frac{K^2}{8}
\int_{\TT}
\left(
g(u_x(t),u_x(t))
\right)^3dx
\nonumber \\
&\phantom{E(u(t))=}-K
\int_{\TT}
\left(
g(u_x(t),\nabla_xu_x(t))
\right)^2dx
\nonumber \\
&\phantom{E(u(t))=}-\frac{3K}{2}
\int_{\TT}
g(u_x(t),u_x(t))
g(\nabla_xu_x(t),\nabla_xu_x(t))dx
\nonumber
\end{align}
are preserved with respect to $t\in [0,T)$.
\end{lemma}
\begin{remark}
\label{remark:remark}
In \cite{NT} and \cite{TN},  
Nishiyama and Tani 
treated \eqref{equation:pde}-\eqref{equation:data} 
in case $N=\mathbb{S}^2$ with $K=1$, 
and proved a time-global existence theorem by using 
the following conserved quantity:
$$
\|u_{xxx}(t)\|^2
-\frac{7}{2}\||u_x(t)||u_{xx}(t)|\|^2
-14\|u_x(t)\cdot u_{xx}(t)\|^2
+\frac{21}{8}\||u_x(t)|^3\|^2,
$$ 
where $\|\cdot\|=\|\cdot\|_{L^2(\TT;\RR^3)}$. 
$E(u(t))$ generalizes the above quantity.
In fact, 
we can check that this quantity is reformulated as
\begin{align}
&\|\nabla_x^2u_x(t)\|_{L^2(\TT;TN)}^2
+
\frac{1}{8}
\int_{\TT}
\left(
g(u_x(t),u_x(t))
\right)^3dx
\nonumber \\
&-
\int_{\TT}
\left(
g(u_x(t),\nabla_xu_x(t))
\right)^2dx
-\frac{3}{2}
\int_{\TT}
g(u_x(t),u_x(t))
g(\nabla_xu_x(t),\nabla_xu_x(t))dx,
\nonumber 
\end{align}
which is just $E(u(t))$ with $K=1$.
\end{remark}
\begin{proof}[Proof of Lemma~\ref{lemma:conserved}]
It is obvious that $\|u_x(t)\|_{L^2(\TT;TN)}^2$ is preserved 
by the same computation as in Section~\ref{section:energy}. 
Hence we omit the proof.
\par 
We consider
\begin{align}
E(u(t))
=&
\|\nabla_x^2u_x(t)\|_{L^2(\TT;TN)}^2
+
A\int_{\TT}
\left(
g(u_x(t),u_x(t))
\right)^3dx
\nonumber \\
&
-B
\int_{\TT}
\left(
g(u_x(t),\nabla_xu_x(t))
\right)^2dx
\nonumber \\
&
-C
\int_{\TT}
g(u_x(t),u_x(t))
g(\nabla_xu_x(t),\nabla_xu_x(t))dx, 
\nonumber
\end{align}
where 
$A=K^2/8$, 
$B=K$, 
$C=3K/2$. 
Since $(N,J,g)$ has a constant sectional curvature $K$ 
as a $C^{\infty}$-manifold, it follows 
for $X,Y,$ and $Z\in \Gamma(u^{-1}TN)$ that
\begin{equation}
R(X,Y)Z
=
K\left\{
g(Y,Z)X-g(X,Z)Y
\right\}.
\label{equation:constant}
\end{equation}
Especially since $\nabla R=0$ holds, 
the term containing $\nabla^pR$, $p\in \mathbb{N}$ does not appear. 
We have only to compute by using \eqref{equation:constant}. 
We make use of the integration by parts 
repeatingly. 
Hence we only show the results of computations. 
A simple computation gives 
\begin{align}
\frac{d}{dt}E(u)
=&
2\int_{\TT}g(\nabla_x^3u_t,\nabla_x^2u_x)dx
\nonumber \\
&
-(2B+2C)\int_{\TT}g(\nabla_xu_x,\nabla_xu_x)g(\nabla_xu_x,u_t)dx
\nonumber \\
&
-(6A+2CK)
\int_{\TT}(g(u_x,u_x))^2g(\nabla_xu_x,u_t)dx
\nonumber \\
&
-(2K+2B+4C)\int_{\TT}g(u_x,\nabla_x^2u_x)g(\nabla_xu_x,u_t)dx
\nonumber \\
&
+(2K-8C)\int_{\TT}g(\nabla_xu_x,u_x)g(\nabla_x^2u_x,u_t)dx
\nonumber \\
&
-(2K+2C)\int_{\TT}g(u_x,u_x)g(\nabla_x^3u_x,u_t)dx
\nonumber \\
&
+(2K-2B)\int_{\TT}g(u_x,\nabla_x^3u_x)g(u_x,u_t)dx
\nonumber \\
&
-(24A-2CK)\int_{\TT}g(u_x,u_x)g(u_x,\nabla_xu_x)g(u_x,u_t)dx
\nonumber \\
&
-(6B-4C)\int_{\TT}g(\nabla_x^2u_x,\nabla_xu_x)g(u_x,u_t)dx.
\nonumber 
\end{align}
We next substitute 
$u_t=a\, \nabla_x^2u_x+J\nabla_xu_x+b\, g(u_x,u_x)u_x$ 
into above and compute by repeating 
integration by parts. 
Then we deduce
\begin{align}
\frac{d}{dt}E(t)
=&
(6K-4C)\int_{\TT}g(\nabla_xu_x, u_x)g(\nabla_x^2u_x, J\nabla_xu_x)dx
\nonumber \\
&-
(2K+4B-4C)\int_{\TT}g(\nabla_xu_x, \nabla_x^2u_x)g(u_x, J\nabla_xu_x)dx
\nonumber \\
&-
(2K-2B)\int_{\TT}g(u_x, \nabla_x^2u_x)g(u_x, J\nabla_x^2u_x)dx
\nonumber \\
&-
(24A-2CK)\int_{\TT}g(u_x, u_x)g(u_x, \nabla_xu_x)g(u_x, J\nabla_xu_x)dx
\nonumber \\
&+
\{-(4K+6B)a+20b\}\int_{\TT}g(u_x, \nabla_x^2u_x)
                           g(\nabla_xu_x, \nabla_x^2u_x)dx
\nonumber \\
&+
\{(4K-6C)a+10b\}\int_{\TT}g(u_x, \nabla_xu_x)
                          g(\nabla_x^2u_x, \nabla_x^2u_x)dx
\nonumber \\
&+
\{(36A+2CK)a-10Cb\}\int_{\TT}g(u_x,u_x)g(u_x, \nabla_xu_x)
                             g(\nabla_xu_x,\nabla_xu_x)dx
\nonumber \\
&+
\{(24A-2CK)a+(-6B+4C)b\}\int_{\TT}\left(g(\nabla_xu_x,u_x)\right)^3dx.
\nonumber 
\end{align}
Since   
$A=K^2/8$, 
$B=K$, 
$C=3K/2$ 
and 
$b=aK/2$, 
a simple computation shows 
\begin{align}
\frac{d}{dt}E(t)
=&
-10(Ka-2b)\int_{\TT}g(u_x, \nabla_x^2u_x)g(\nabla_x^2u_x, \nabla_xu_x)dx
\nonumber \\
&-5(Ka-2b)\int_{\TT}g(u_x, \nabla_xu_x)g(\nabla_x^2u_x, \nabla_x^2u_x)dx
\nonumber \\
&+\frac{15}{2}K(Ka-2b)\int_{\TT}g(u_x,u_x)g(u_x, \nabla_xu_x)
                                g(\nabla_xu_x,\nabla_xu_x)dx
\nonumber \\
=&0,
\nonumber
\end{align}
which  completes the proof. 
\end{proof}
\begin{proof}[Proof of Theorem~\ref{theorem:meo}]
Let 
$u\in C([0,T);H^{m+1}(\TT;N))$ 
be the time-local solution of  
\eqref{equation:pde}-\eqref{equation:data} 
which exists on the maximal time interval $[0,T)$. 
If $T=\infty$, 
Theorem~\ref{theorem:meo} holds true. 
Thus we only need to consider the case $T<\infty$. 
From Lemma~\ref{lemma:conserved}, 
we know that
\begin{equation}
\|u_x(t)\|_{L^2(\TT;TN)}^2
=
\|u_{0x}\|_{L^2(\TT;TN)}^2,
\quad
E(u(t))=E(u_0).
\label{equation:q0}
\end{equation}
Hence it follows that
\begin{align}
\|\nabla_x^2u_x(t)\|_{L^2(\TT;TN)}^2
=&
E(u_0)-\frac{K^2}{8}
\int_{\TT}
\left(
g(u_x(t),u_x(t))
\right)^3dx
\nonumber \\
&+
K
\int_{\TT}
\left(
g(u_x(t),\nabla_xu_x(t))
\right)^2dx
\nonumber \\
&+
\frac{3K}{2}
\int_{\TT}
g(u_x(t),u_x(t))
g(\nabla_xu_x(t),\nabla_xu_x(t))dx
\nonumber \\
\leqslant&
E(u_0)
+
C|K|
\|u_x(t)\|_{L^{\infty}(\TT;TN)}^2
\|\nabla_xu_x(t)\|_{L^2(\TT;TN)}^2.
\nonumber 
\end{align}
Note here the Sobolev inequality 
and the Gagliardo-Nirenberg inequality 
of the form
\begin{align}
\|u_x(t)\|_{L^{\infty}(\TT;TN)}^2
&\leqslant
C
\|u_x(t)\|_{L^2(\TT;TN)}
(\|u_x(t)\|_{L^2(\TT;TN)}
+
\|\nabla_xu_x(t)\|_{L^2(\TT;TN)}),
\label{equation:so}
\\
\|\nabla_xu_x(t)\|_{L^2(\TT;TN)}
&\leqslant
C
\|\nabla_x^2u_x(t)\|_{L^2(\TT;TN)}^{1/2}
\|u_x(t)\|_{L^2(\TT;TN)}^{1/2}
\label{equation:gn}
\end{align}
hold. 
See e.g., \cite[Lemma~1.~3. and 1.~4.]{Koiso2}. 
From \eqref{equation:q0}, \eqref{equation:so} and \eqref{equation:gn},  
we deduce
\begin{align}
&\|\nabla_x^2u_x(t)\|_{L^2(\TT;TN)}^2
\nonumber \\\
&\quad
\leqslant
E(u_0)
+
C|K|\|u_{0x}\|_{L^2(\TT;TN)}
\nonumber \\
&\qquad \qquad \qquad
\times
\left(
\|u_{0x}\|_{L^2(\TT;TN)}
+
\|u_{0x}\|_{L^2(\TT;TN)}^{1/2}
\|\nabla_x^2u_x(t)\|_{L^2(\TT;TN)}^{1/2}
\right)
\nonumber \\
&\qquad \qquad \qquad
\times
\|u_{0x}\|_{L^2(\TT;TN)}
\|\nabla_x^2u_x(t)\|_{L^2(\TT;TN)}
\nonumber \\
&\quad
\leqslant
C(N,\|u_{0x}\|_{H^2(\TT;TN)})
(1+\|\nabla_x^2u_x(t)\|_{L^2(\TT;TN)}^{3/2}).
\nonumber
\end{align}
Thus $X=X(t)=\|\nabla_x^2u_x(t)\|_{L^2(\TT;TN)}$ 
satisfies $X^2\leqslant C(1+X^{3/2})$,  
which implies
\begin{align}
\sup_{t\in [0, T)}\|\nabla_x^2u_x(t)\|_{L^2(\TT;TN)}
&\leqslant
C(N,\|u_{0x}\|_{H^2(\TT;TN)})
\label{equation:q1}
\intertext{
for some $C=C(N,\|u_{0x}\|_{H^2(\TT;TN)})>0$. 
Interpolating \eqref{equation:q0} and \eqref{equation:q1} 
we have}
\sup_{t\in [0,T)}
\|u_x(t)\|_{H^2(\TT;TN)}
&\leqslant
C(N,\|u_{0x}\|_{H^2(\TT;TN)}).
\nonumber \\
\intertext{Since we obtain the $H^2(\TT;TN)$-boundness 
of $u_x$,}
\sup_{t\in [0,T)}
\|u_x(t)\|_{H^m(\TT;TN)}
&\leqslant
C(N,\|u_{0x}\|_{H^m(\TT;TN)})
\nonumber
\end{align}
follows as in the proof of Theorem~\ref{theorem:eo}. 
Hence, for small $0<\sigma <T$,
there exists  $T_0>0$  
and a  time-local solution $u_1$ of 
\eqref{equation:pde}-\eqref{equation:data} 
on the time interval $[0,T_0)$ 
with initial data 
$u_1(0, x)=u(T-\sigma, x)$.
From the  uniform estimate of 
$\|u_x(t)\|_{H^2(\TT;TN)}$ 
on $[0,T)$, we see
$T_0$ does not depend on $\sigma$. 
Thus, by choosing  $\sigma$ small enough, 
we have 
$T-\sigma+T_0 >T$.  
By the uniqueness theorem, 
we know 
$u_1(t,x)=u(T-\sigma+t,x)$ 
for any $t\in[0,T_0)$. 
Thus $u$ can be extended  to the time interval 
$[0,T-\sigma+T_0)$, 
which contradicts the maximality of $T$. 
\end{proof}
{\bf Acknowledgement.} \\
The author expresses gratitude to Hiroyuki Chihara  
for several discussions and valuable advice.

%
%
\end{document}